\documentclass[11pt]{article}
\usepackage[sort&compress,numbers,comma,square]{natbib}
\RequirePackage{amsmath, amsthm, amssymb, amsfonts, enumitem,
  verbatim, xspace, ifthen}
\RequirePackage[active]{srcltx}
\RequirePackage[colorlinks=true, linkcolor=blue, citecolor=blue]{hyperref}
\long\def\comment#1{}

\newtheorem{theorem}{Theorem}[section]
\newtheorem{prop}[theorem]{Proposition}
\newtheorem{cor}[theorem]{Corollary}
\newtheorem{lemma}[theorem]{Lemma}

\theoremstyle{definition}
\newtheorem{example}[theorem]{Example}
\numberwithin{equation}{section}

\def\Abs#1{\left|#1\right|}
\def\AND{,\,}
\def\Bb#1{\mathbb{#1}}
\def\Cbr#1{\left\{#1\right\}}
\def\cf#1{\mathbf{1}\{#1\}}
\def\chf{\varphi}
\def\Coms{\mathbb{C}}
\def\conv#1{\stackrel{#1}{\rightarrow}}
\def\cspace{\mathrm{csp}}
\def\DA{\mathcal{D}}
\def\dd{\mathrm{d}}
\def\diag{\mathrm{diag}}
\def\dpois{\mathrm{Poisson}}
\def\eno#1#2{#1_1, \ldots, #1_{#2}}
\def\exf#1{e^{#1}}
\def\exfi#1{\exf{\iunit{#1}}}

\def\Flr#1{\left\lfloor#1\right\rfloor}
\def\ft#1{\widehat{#1}}
\def\Grp#1{\left(#1\right)}
\def\gv{\,|\,}
\def\Id{\mathrm{Id}}
\def\intii{\int_{-\infty}^\infty}
\def\Ints{\mathbb{Z}}
\def\inum#1{#1_1, #1_2, \ldots}                 
\def\ip#1#2{\langle{#1},{#2}\rangle}
\def\iunit{\mathrm{i}}

\def\levy{L\'evy\xspace}

\def\lspan{\mathrm{span}}
\def\mean{\mathbb{E}}
\def\Nats{\mathbb{N}}
\def\nth#1{\frac{1}{#1}}
\def\pr#1{\mathbb{P}\{#1\}}
\def\Pr#1{\mathbb{P}\Cbr{#1}}
\def\Pseq#1#2{F^{*#1}(#2)}
\def\Rats{\mathbb{Q}}
\def\Reals{\mathbb{R}}
\def\RV{\mathcal{R}}
\def\rx{\epsilon}

\def\Sbr#1{\left[#1\right]}
\def\sign{\text{sgn}}
\def\Norm#1{\|{#1}\|}
\def\Sp#1{\sp{(#1)}}

\def\ssp{\Reals^d}
\def\sumoi#1{\sum_{#1=1}^\infty}
\def\th{\text{-th}}
\def\toi{\to\infty}

\def\Lsup{\varlimsup}
\def\Iff{\Longleftrightarrow}   

\def\Bing{bingham:89:cup}
\def\Doney{doney:97:ptrf}
\def\Gar{garsia:63:cmh}
\def\Petr{petrov:95:oxford}
\def\Rva{rvaceva:62:stmsp}
\def\Sato{sato:99:cup}
\def\Spit{spitzer:76:sv-nh}
\def\Stone{stone:67:bsmsp}
\def\Will{williamson:68:pjm}

\makeatletter
\def\iid{i.i.d\@ifnextchar.{}{.\ }}
\makeatother

\setlist{label=({\arabic*}), parsep=.25ex, topsep=1ex, itemsep=0.15ex,
  leftmargin=0ex, itemindent=1.7\parindent}

\begin{document}
\begin{center}
  \textbf{\MakeUppercase{
      On Multivariate Strong Renewal Theorem
    }
  } \\[1ex]
  \normalsize
  Zhiyi Chi\\
  Department of Statistics,
  University of Connecticut \\
  Storrs, CT 06269, USA.
  E-mail: zhiyi.chi@uconn.edu\\[.8ex]
  \today
\end{center}
  
\begin{abstract}
  This paper takes the so-called probabilistic approach to the Strong
  Renewal Theorem (SRT) for multivariate distributions in the domain of
  attraction of a stable law.  A version of the SRT is obtained that
  allows any kind of lattice-nonlattice composition of a distribution.  
  A general bound is derived to control the so-called ``small-$n$
  contribution'', which arises from random walk paths that have a
  relatively small number of steps but make large cumulative moves.
  The asymptotic negligibility of the small-$n$ contribution is 
  essential to the SRT.  Applications of the SRT are given, including
  some that provide a unified treatment to known results but with
  substantially weaker assumptions.

  \medbreak\noindent
  
  \medbreak\noindent
\end{abstract}

\section{Introduction} \label{s:intro}
For a probability distribution $F$ on $\ssp$, the Strong Renewal
Theorem (SRT) is said to hold if
\begin{align} \label{e:SRT}
  \frac{|x|^d}{A(|x|)} U(x+E)
  \to
  g(x/|x|) u(E)
\end{align}
uniformly in a certain sense as $|x|\toi$, where $U = \sumoi n F^{*n}$
is the renewal measure with $F^{*n}$ the $n$-fold convolution of $F$,
$A(\cdot)>0$ is a function on $(0,\infty)$, $E$ is some ``nice'' set,
$x + E$ denotes $\{x + y: y\in E\}$, $g(\cdot)\not\equiv 0$ is a
function on $\{\omega\in\ssp: |\omega| =1\}$, and finally, $u$ is a
nonzero $\sigma$-finite measure on $\ssp$; see Theorem \ref{t:SRT} for
precise explanation.  The definition extends the one in \cite {\Will}
that only considers $F$ on $\Ints^d$.  In \cite{\Will}, $x$ stays in
$\Ints^d$ and $E = \{0\}$.  However, in general, $x$ can take any
value in $\Reals^d$ and $E$ has to depend on the lattice-nonlattice
composition of $F$.

There are two main approaches to the SRT.  One is based on Fourier
analysis of the renewal measure \cite {erickson:70:tams,
  erickson:71:bams, \Gar, doney:66:plms, uchiyama:98:plms, \Will}.
The other is the so-called probabilistic approach \cite{\Spit, \Doney,
  vatutin:13:tpa, chi:15:aap, chi:14:tr, \Will, doney:15:arxiv,
  caravenna:15:arxiv}.  It is based on the realization that the two
partial sums that comprise the renewal measure,
\begin{align} \label{e:big-small-n}
  \sum_{n\ge A(\delta |x|)} F^{*n}(x+E)
  \qquad\text{and}\quad
  \sum_{n<A(\delta |x|)} F^{*n}(x+E)
\end{align}
with $\delta>0$ an arbitrary fixed number, are essentially different
and hence should be tackled in different ways.  The partial sums in
\eqref{e:big-small-n} will be referred to as the ``big-$n$'' and
``small-$n$'' contributions, respectively.  In general, the big-$n$
contribution can be dealt with using Local Limit Theorems (LLTs),
essentially yielding the limit \eqref{e:SRT} provided it exists \cite
{\Spit, \Gar, erickson:70:tams, chi:15:aap,  \Will}.  In contrast,
without additional conditions, the small-$n$ contribution often fails
to converge, hence ruling out the existence of the limit \cite {\Gar,
  \Will, vatutin:13:tpa}.  Recently, to control the small-$n$
contribution when $d=1$, integral criteria were proposed \cite
{chi:15:aap, chi:14:tr}.  This paper extends the idea in \cite
{chi:15:aap, chi:14:tr} to the multivariate case.  As in previous
works, it investigates the SRT for $F$ in the domain of attraction
{\em without centering\/} of a {\em  nondegenerate\/} stable law.  By
definition, there are $a_n\in\Reals$ such that $F^{*n}(a_n \dd x)$
weakly converges to an $\alpha$-stable law not concentrated in any
linear manifold of dimension $d-1$.  Denote this by $F\in
\DA_0(\alpha)$.  To establish the SRT for $F$, the small-$n$
contribution is approached by analyzing various subsets of random walk 
paths, in particular components of the paths at different scales.  In
addition to being quite easily applicable, the resulting SRT gives a
unified  treatment to many known results, sometimes with substantially
weaker assumptions.

It should be remarked that for $d>1$, a more general type of stability
can be defined, namely operator-stability (cf.~\cite
{sharpe:69:tams}).  Characterizations of domain of attraction for 
operator-stability as well as the corresponding LLTs are known (cf.\
\cite{resnick:79:jma, hudson:83:ap, hahn:85:zw, griffin:86:ap,
  doney:91:jma}).  Since operator-stability has found applications,
e.g., in the study on the ladder height and ladder epoch of random
walks in $\Reals$ \cite{doney:93:ptrf, greenwood:82b:zw}, it is of
interest to consider the related SRT.  This topic is beyond the scope
of the paper.

During the revision of the paper, sufficient and necessary conditions
for the SRT in the univariate case were announced \cite
{caravenna:16:arxiv}.  The key to the new result is a new local large
deviation bound for $F^{*n}$.  An extension to the multivariate case
will be interesting in future work.

Section \ref{s:main} presents the main result of the paper, which is a
multivariate SRT in Theorem \ref{t:SRT}.  The SRT is preceded by a
result on the lattice-nonlattice composition of a distribution, which
is an issue unique to the multivariate case and has to addressed in
order to formulate the SRT properly.  Applications of the SRT are also
presented in the section.  The proof of Theorem \ref{t:SRT} is
outlined in Section \ref{s:outline}.  It is shown in this section that
the theorem is a consequence of Theorems \ref {t:big-n} and \ref
{t:small-n} that deal with the big-$n$ and small-$n$ contributions,
respectively.  As a preparation for their proofs, Section \ref
{s:bounds} derives bounds for the \levy concentration and local large
deviation of $F^{*n}$.  Then Theorems \ref{t:big-n} and
\ref{t:small-n} are proved in Sections \ref {s:big-n}--\ref
{s:small-n}, respectively. Section \ref{s:proof-props} collects proofs
of minor results on the SRT.  The lattice-nonlattice composition is
proved in Appendix \ref{s:lattice}.

The rest of this section fixes notation.  For $a$, $b\in\Reals$, denote
$a\vee b = \max(a,b)$, $a\wedge b = \min(a, b)$, and $a_+ = a\vee 0$.
For $x\in\ssp$, denote by $|x|$ its Euclidean norm and $\Norm x =
\max_i |t_i|$ its sup-norm, where $t_i$ are the coordinates of $x$.
Denote $B_d = \{x\in\ssp: |x| \le 1\}$, $S^{d-1} = \{x\in\ssp:
|x|=1\}$, and $I_d = [0,1)^d$.  For $\Lambda\subset \Reals$ and $D$,
$E\subset \ssp$, denote $a D = \{a y: y\in D\}$, $\Lambda x =
\{\lambda x: \lambda \in \Lambda\}$, $x+D = \{x+y: y\in D\}$, and
$D+E=\{y+z: y\in D, z\in E\}$.  Denote $M\in \Lambda^{m\times d}$ if
$M$ is an $m\times d$ matrix of elements in $\Lambda$, and $M D = \{M
y: y\in D\}$.  Denote by $\diag(\eno a n)$ the diagonal matrix with
the $i\th$ diagonal element being $a_i$, and $\Id_n$ the $n\times n$
identity matrix.  For a linear subspace $V$ of $\ssp$, denote by
$\pi_V$ the projection onto $V$.    If $f\in L^1(\ssp)$, denote
$\ft f(t) = \int\exfi{\ip t x} f(x)\, \dd x$.

For functions $f$ and $g$, $f(x)=O(g(x))$, $f(x)\ll g(x)$, and
$g(x)\gg f(x)$ all mean $|f(x)| \le C |g(x)|$ for some constant $C>0$,
and $f(x)\asymp g(x)$ means $g(x)\ll f(x)\ll g(x)$.  If $C$ depends on
parameters $\eno a k$, when it is necessary to emphasize the
dependence, denote $f(x)=O_{\eno a k}(g(x))$, $f(x)
\ll_{\eno a n} g(x)$, or $g(x)\gg_{\eno a n} f(x)$.  By $f(x) =
o_{\eno a k}(g(x))$ as $x\toi$ it means there is a function $M(\rx) =
M(\rx; \eno a k)$, such that $|f(x)|\le \rx |g(x)|$ for all $x\ge
M(\rx)$.

\section{Main results}  \label{s:main}
\subsection{Lattice-nonlattice composition of distribution}
It is well known that the SRT, in particular, the big-$n$ contribution
involved, has to be handled differently for lattice distributions and
nonlattice ones \cite {gnedenko:54:wzh, \Rva, stone:65:ams, \Stone,
  \Bing}.  Recall that if a distribution is concentrated on $a+\Gamma$
for some $a\in \ssp$ and lattice $\Gamma\subset\ssp$, then the
distribution as well as any random variable following it is said to be
lattice.  By definition, $\Gamma$ is an additive subgroup of $\ssp$
with no cluster points.  For $d>1$, a complication is that a
distribution may be jointly lattice and nonlattice, so it is necessary
to first know its lattice-nonlattice composition in order to establish
the SRT.  The lattice-nonlattice composition of a nondegenerate
distribution is characterized by the next result that will be proved
in Appendix \ref{s:lattice}.  Recall that two integers are coprime if
their greatest common divisor is 1, and $\eno a n\in\Reals$ are
rationally independent if for $\eno m n\in\Ints$, $\sum m_i a_i
\in\Ints\Iff$ all $m_i=0$ (\cite {petersen:83:cup}, p.~51).

\begin{prop} \label{p:lattice-decomp}
  Let $X$ be a nondegenerate random variable in $\ssp$.  Denote by 
  $\chf_X(u) = \mean[\exfi{\ip u X}]$ its characteristic function.
  Then there exist a linear subspace $V\subset\ssp$, a nonsingular
  matrix $T\in\Reals^{d \times d}$, and integers $0\le \nu\le r\le
  d$ and $q\ge 1$ with the following properties.
  \begin{enumerate}
  \item $\pi_V(X)$ is lattice, and $|\chf_X(2\pi v)|<1$ for $v\in\ssp
    \setminus V$.
  \item $|\chf_{TX}(2\pi u)|=1 \Iff u\in \Ints^r \times \{0\}$, where
    $0$ is the zero vector in $\Ints^{d-r}$.
  \item Let $TX = (Y, Z)$ with $Y\in\Reals^r$ and $Z\in\Reals^{d-r}$.
    Then $\pr{Y\in\beta + \Ints^r} = 1$ for $\beta = (0, \ldots, 0,
    \beta_\nu, \ldots, \beta_r)$,  where $\beta_\nu = p/q\in\Rats$
    with $0\le p<q$ being coprime and $\beta_{\nu+1}$, \ldots,
    $\beta_r\in (0,1)\setminus \Rats$ are rationally independent.
  \end{enumerate}
  Furthermore, $V$,   $r$, $\nu$, and $q$ with above properties are
  unique, and $r=\dim(V)$.
\end{prop}

\begin{proof}[Remark]\
  \begin{enumerate}
  \item In the decomposition, $\beta_\nu=0\Iff p=0$ and $q=1$.
    
  \item It was claimed in \cite{\Stone} that according to p.~64--75 of
    \cite {\Spit}, $X$ can always be linearly transformed into a
    nondegenerate $\xi\in\ssp$, such that for some $\beta\in\Reals^r$,
    $\chf_\xi(2\pi u) = \exp[2\pi\iunit \ip \beta v]$ if $u = (v, 0)
    \in \Ints^r\times \{0\}$ and $|\chf_\xi(2\pi u)|<1$ otherwise.
    However, $X$ is assumed to be $\Ints^d$-valued in \cite{\Spit}, so
    it is unclear how the claim was obtained.  Moreover, the SRT
    requires detailed information about $\beta$, so the claimed
    transformation is insufficient.  \qedhere
  \end{enumerate}
\end{proof}

Denote 
\begin{align} \label{e:lattice-l-w}
  l_0 = (0, \ldots, 0, \beta_\nu), \quad
  w_0 = (\beta_{\nu+1}, \ldots, \beta_r).
\end{align}
By Proposition \ref{p:lattice-decomp}, if $Y$ is partitioned as $(L,
W) \in \Reals^\nu \times \Reals^{r-\nu}$ so that
\begin{align} \label{e:normalization}
  TX = (Y, Z) = (L,W,Z),
\end{align}
then $(L, W)$ and $Z$ are the lattice and nonlattice components of
$X$, respectively.  Meanwhile, $L$ and $(W, Z)$ are the arithmetic and
nonarithmetic components, respectively.  The dimensions of the
components are unique, and the number $q\in\Nats$ such that  $D L$ is
$\Ints^\nu$-valued is unique, where
\begin{align} \label{e:cell0}
  D = \diag(1,\ldots, 1, q)\in \Ints^{\nu\times\nu}.
\end{align}
However, Proposition \ref{p:lattice-decomp} does not say $\beta_\nu$,
\ldots, $\beta_r$ are unique.

The SRT for an aperiodic random walk is studied in \cite{\Will}.  A
random variable $\xi\in \Ints^\nu$ is said to be aperiodic if for any
nonrandom $t\in \Reals^\nu$, $\pr{\ip t\xi\in \Ints}=1 \Iff t\in
\Ints^\nu$, and is said to be strongly aperiodic if for any nonrandom 
$t\in \Reals^\nu$ and $c\in\Reals$, $\pr{\ip t\xi \in c + \Ints}=1\Iff
t\in\Ints^\nu$ and $c\in\Ints$ (cf.\ \cite{\Spit}, T7.1, P7.8).  The
following result will be proved in Appendix \ref {s:lattice}.  Recall
that a matrix $K\in \Ints^{\nu  \times \nu}$ has an inverse in
$\Ints^{\nu\times \nu}\Iff |\det K|=1$, in which case $K^{-1}$ is
$\det K$ times the adjugate of $K$.
\begin{prop} \label{p:aperiod}
  \
  \begin{enumerate}
  \item Let $\xi\in\Ints^\nu$ be nondegenerate.  Then $\xi$ is
    aperiodic $\Iff$ there are $K\in\Ints^{\nu\times\nu}$ with $|\det
    K|=1$ and coprime integers $0\le p<q$, such that $\xi
    =K^{-1}(D\zeta + p e_\nu)$, where $\zeta$ is strongly aperiodic,
    $D$ is defined in \eqref{e:cell0}, and $e_\nu = (0,\ldots,0,1)$ is
    the $\nu\th$ standard base vector of $\Ints^\nu$.

  \item For the $L$ and $D$ in \eqref{e:normalization}--\eqref
    {e:cell0}, $L - \beta_\nu e_\nu$ is strongly aperiodic and $D L$
    is aperiodic.
  \end{enumerate}
\end{prop}

\subsection{A sufficient condition for the SRT}
The lattice-nonlattice composition in Proposition \ref
{p:lattice-decomp} allows the SRT to be formulated properly.  The 
next SRT is the main result.  It also implies certain property of
the limiting law involved.  Recall that a nondegenerate stable law has
an infinitely differentiable density with all derivatives vanishing at
$\infty$ (\cite{\Sato}, Example 28.2).

\begin{theorem}[SRT] \label{t:SRT}
  Let $F\in \DA_0(\alpha)$ be a distribution on $\ssp$ with
  $0<\alpha\vee 1<d$.  Let $\psi$ be the density of the limiting
  stable law of $F^{*n}(a_n\dd x)$, where $a_n>0$ is a sequence of
  norming constants.  Let $A(s)$ be any function regularly varying at
  $\infty$ such that $A(a_n)/n\to 1$ as $n\toi$.  Define
  \begin{align} \label{e:partial-radial}
    \varrho_s(\omega)
    =
    \alpha q^{-1} \int_0^{1/s} \psi(u\omega) u^{d-\alpha-1}\,\dd u,
    \quad \omega\in S^{d-1},\ s>0.
  \end{align}
  Let $T\in \Reals^{d\times d}$ be nonsingular such that $T X = (L, W,
  Z)$ as in \eqref{e:normalization}.  Fixing $\Upsilon\in \Ints^{\nu
    \times \nu}$ with $|\det\Upsilon| = 1$, define
  \begin{align} \label{e:cell}
    \Delta_h = (D^{-1} \Upsilon I_\nu)\times (h I_{d-\nu}), \quad h>0,
  \end{align}
  where $D$ is given in \eqref{e:cell0}.   Define
  \begin{align} \label{e:integral-K}
    K(t,a,\eta,h) = 
    \int_{|z|<\eta |t|}
    F(t - z + h I_d) \exf{-|z|/a}\, \dd z
  \end{align}
  for $t\in\ssp$, $a>0$, $r>0$, and $h>0$.  Define
  \begin{align} \label{e:kappa}
    \kappa = \Flr{d/\alpha}.
  \end{align}
  Then, if there are $\theta\in (0, 1/\kappa)$ and $\eta>0$ such that
  \begin{align} \label{e:SRT-tail}
    \lim_{\delta\to 0} \Lsup_{s\toi} 
    \frac{s^d}{A(s)}\sum_{n\le A(\delta s)} 
    n a^{-d}_n\sup_{|t|>\theta s} K(t, a_n, \eta, h)
    = 0,
  \end{align}
  then the following convergence holds
  \begin{gather} \label{e:uniform-SRT}
    \lim_{s\toi} \sup_{\omega\in S^{d-1}}
    \Abs{
      \frac{s^d}{A(s)} U(s\omega + T^{-1} \Delta_h)
      -
      \frac{h^{d-\nu}\varrho_0(\omega)}{|\det T|}
    }
    = 0
  \end{gather}
  and moreover,
  \begin{gather} \label{e:uniform-SRT2}
    \lim_{\delta\to 0}\sup_{\omega\in S^{d-1}}
    |\varrho_0(\omega) -\varrho_\delta(\omega)| = 0.
  \end{gather}
\end{theorem}

\begin{proof}[Remark]
  Eq.~\eqref{e:uniform-SRT} makes clear what the uniform convergence
  in \eqref{e:SRT} means and will be referred to as the SRT.  For
  $\alpha=2$, since $\psi$ is a normal density, the uniform
  convergence in \eqref {e:uniform-SRT2} holds without assuming
  \eqref{e:SRT-tail}.  However, for $\alpha\in (0,2)$ and $d>1$, it
  may fail to hold.  Indeed, following upon Example 5-B of
  \cite{\Will}, given $1\le k<d$, if $\psi$ is the density of $(\eno
  \xi d)$, where $\xi_i\in \Reals$ are \iid symmetric $\alpha$-stable
  with $\alpha \le (d-k)/(k+1)$ and density $g$, then for any
  $\omega\in S^{d-1}$ with at most $k$ nonzero coordinates,
  $\psi(s\omega) s^{d-\alpha -1}  = g(\omega_1 s) \cdots g(\omega_d s)
  s^{d-\alpha-1}\gg s^{d-(k+1)(1+\alpha)} \gg s^{-1}$ as $s\toi$,
  giving $\rho_0(\omega)=\infty$.
\end{proof}

The condition \eqref{e:SRT-tail} is often easy to check.  From Theorem
\ref{t:SRT}, the following SRT follows.  When $F$ is concentrated on
$\Ints^d$, the same result was established in \cite{\Will} but with a
very different argument.  Unlike \cite{\Will}, the proof given in
Section \ref{s:proof-props} applies to $F$ with any lattice-nonlattice
composition.
\begin{theorem} \label{t:w-SRT}
  If $1\le d/2<\alpha\in (0,2)$ (so $d=2$ or 3), then the SRT \eqref
  {e:uniform-SRT} holds for any $F\in \DA_0(\alpha)$ and \eqref
  {e:uniform-SRT2} holds for any nondegenerate $\alpha$-stable law on
  $\ssp$.
\end{theorem}

The above result does not cover $\alpha=2$.  For this case, the next
result provides weaker conditions than \cite{\Spit}, P26.1, and \cite
{uchiyama:98:plms}.
\begin{prop} \label{p:normal}
  Let $F\in \DA_0(2)$ and $X\sim F$.  Denote $q_X(s) = \pr{|X|>s}$.
  Then the SRT \eqref{e:uniform-SRT} holds for $F$ in each of the
  following cases.
  \begin{enumerate}
  \item $d=3$.
  \item $d=4$ and
    \begin{align*}
      q_X(s) \int_1^s u^{-5} A(u)^2 \,\dd u = o(A(s)/s^4), \quad
      s\toi.
    \end{align*}
  \item $d\ge 5$ and $q_X(s) = o(s^{2-d})$.
  \end{enumerate}
\end{prop}
For $d=3$, the SRT for $X\sim F\in\DA_0(2)$ is established in \cite
{\Spit}, P26.1, under the condition $\sigma^2 = \mean |X|^2 < \infty$.
For $d=4$, the SRT is established in \cite{uchiyama:98:plms} under the
condition $\mean |X|^2(\ln |X|)_+<\infty$, which implies $\sigma^2 <
\infty$.  However, if $\sigma^2<\infty$, then by the Central
Limit Theorem, $A(s)/s^2\asymp 1$ and the display in (2) reads
$q_X(s) = o(1/(s^2 \ln s))$, which is a weaker condition.  In Example
\ref {ex:normal} below, it is shown that even $\sigma^2<\infty$ is not
necessary.  Finally, for $d\ge 5$, the SRT is established in \cite
{uchiyama:98:plms} under the condition $\mean|X|^{d-2} < \infty$.
Clearly, the condition in (3) is weaker.  In Example \ref{ex:normal2},
it will be seen that the condition and even $\sigma^2<\infty$ is not
necessary.

\begin{example} \label{ex:normal}
  Let $X\in\Reals^4$ be spherically symmetric with $q_X(s) \asymp
  1/(s^2 \ln s)$.  Put $V_X(s) = \mean[|X|^2 \cf{|X|\le s}]$.  By
  $V_X(s) \asymp \int_1^s u q_X(u) \,\dd u \asymp \int_1^s (u\ln
  u)^{-1} \,\dd u \asymp \ln \ln s$, $\mean |X|^2 = \infty$.  On the
  other hand, since $A(s) \sim s^2/V_X(s) \asymp s^2/\ln \ln s$, then
  $q_X(s) = o(1/A(s))$ and so $X\in \DA_0(2)$ (cf.\ \eqref
  {e:prelim-A} and \cite{\Rva}, Th.~4.1).  Meanwhile, $\int_1^s u^{-5}
  A(u)^2\,\dd u \asymp \ln s/(\ln \ln s)^2$.   As a result, the
  condition in Proposition \ref{p:normal}(2) is satisfied and the SRT
  holds.  \qed
\end{example}

Next consider a multivariate version of a result in \cite{chi:15:aap,
  chi:14:tr}.  Define
\begin{align*}
  \phi(x) = |x|^d F(x+ h I_d) A(|x|).
\end{align*}
The classical condition $\sup\phi<\infty$ for $d=1$ played a
critical role in several works \cite{\Will, \Doney, vatutin:13:tpa}.
\begin{theorem} \label{t:bounded-ratio}
  Let $\alpha\in (0,2]\cap (0, d/2)$.  Suppose there are $T\ge 0$
  and $\eta>0$ such that
  \begin{align}  \label{e:bounded-ratio}
    \sup_\omega \int_{|z|<\eta s} [\phi(s\omega-z)-T]_+\,\dd z =
    o(A(s)^2), \quad s\toi,
  \end{align}
  then for any $\theta>0$ and $\delta>0$,
  \begin{align}\label{e:bounded-ratio-K}
    \sum_{n\le A(\delta s)} n a^{-d}_n
    \sup_{|t|\ge \theta s} K(t, a_n, \eta, h)
    =
    [o(1) + \delta^{2\alpha}]\frac{A(s)}{s^d}, \quad s\toi
  \end{align}
  and consequently the SRT \eqref{e:uniform-SRT} holds for $F$ and
  \eqref{e:uniform-SRT2} for the limiting stable law of $F^{*n}(a_n
  \dd x)$.
\end{theorem}

\begin{example} \label{ex:normal2}
  As an application of Theorem \ref{t:bounded-ratio}, it can be
  shown that for $d\ge 5$, the condition in Proposition
  \ref{p:normal}(3), i.e., $q_X(s) = o(s^{2-d})$, is not necessary for
  the SRT.  Indeed even $\mean |X|^2 < \infty$ is not necessary.  Let
  $X$ have density $f(x)= c(1+|x|)^{-d-2}$, where $c>0$ is a constant.
  Put $V_X(s) = \mean[|X|^2 \cf{|X|\le s}]$.  Then $V_X(s) =
  \int_{|x|\le s} |x|^2 f(x)\,\dd x \asymp \ln s$ as $s\toi$, so
  $\mean |X|^2=\infty$.  However, by $s^2 q_X(s) = s^2 \int_{|x|\ge s}
  f(x)\,\dd x = O(1)$, the law of $X$ is in $\DA_0(2)$ (\cite{\Rva},
  Th.~4.1).  Moreover, since $A(s) \sim s^2/V_X(s)$ (cf.\
  \eqref{e:prelim-A}), $\phi(x) \asymp_h |x|^d f(x) A(|x|) $ is
  bounded.  Then by Theorem \ref {t:bounded-ratio}, the SRT holds.
  \qed
\end{example}

\begin{example}  \label{ex:w-5B}
  Let $\xi\in\Reals$ be symmetric $\alpha$-stable with $\alpha\in
  (0,2)$ and $X=(\eno X d)$ with $X_i$ \iid $\sim \xi$.  From the
  remark for Theorem \ref{t:SRT}, if $\alpha \le (d-1)/2$, then
  $\varrho_0(e_i)=\infty$.  Therefore, both \eqref{e:uniform-SRT} and 
  \eqref{e:uniform-SRT2} fail to hold.  The goal here is to show that
  if $\alpha> (d - 1)/2$ (so $d\le 4$), then both hold.  Fix
  $\theta>0$ and $0< \eta < 1/(10 d)$.  Given $t = (\eno t d)$, put
  $x = t_i$, where $|t_i| = \max |t_j|$.  Then $|x|\ge |t|/d$ and so
  \begin{align}
    K(t,a, \eta, h)
    &=
    \int_{|z|<\eta|t|}
    \prod_{j=1}^d \pr{X_i\in t_j -z_j + h I_1}
    \exf{-|z|/a}\,\dd z
    \nonumber\\
    &\le
    \int_{|z_i|<\eta d |x|}
    \prod_{j=1}^d \pr{X_i\in t_i -z_i + h I_1}
    \exf{-|z_i|/a}\,\dd z
    \nonumber\\
    &=
    h^{d-1}
    \int_{|u|<\eta d |x|}
    \pr{\xi\in x -u + h I_1} \exf{-|u|/a}\,\dd u
    \label{e:K-ind-comp}
  \end{align}
  For $x\in\Reals$, $\pr{\xi \in x + h I_1} \ll_h |x|^{-\alpha-1}$.
  On the other hand, for $|t|\gg 1$, if $|x|\ge |t|/d$ and $|u| \le
  \eta d|x| < |x|/10$, then $|x - u| \asymp |t|$.  Then
  \begin{align*}
    K(t, a, \eta, h) \ll_h |t|^{-\alpha-1} \intii \exf{-|u|/a} \dd
    u \ll a |t|^{-\alpha-1}.
  \end{align*}
  Since $a_n \sim n^{1/\alpha}$ and $A(t) \sim t^\alpha$,
  \begin{align*}
    \sum_{n\le A(\delta s)} n a^{-d}_n
    \sup_{|t|>\theta s} K(t,a_n, \eta, h)
    \ll_h s^{-\alpha-1} \sum_{n\le A(\delta s)} n^{1-(d-1)/\alpha}
    \ll \delta^{2\alpha-d+1} A(s)/s^d.
  \end{align*}
  Then by Theorem \ref{t:SRT}, the SRT holds for $X$.  \qed
\end{example}

As seen earlier, for a strictly $\alpha$-stable distribution $G$ on
$\ssp$ with $\alpha\in (0,2)$, the uniform convergence of
$\varrho_\delta$ to $\varrho_0$ may fail to hold if $d>1$, where
$\varrho_\delta$ is defined in \eqref {e:partial-radial}.  As an
application of Theorem \ref{t:SRT}, a sufficient condition for the
uniform convergence will be provided next.  Let $\psi$ be the density
of $G$.  By Theorem 14.10 in \cite{\Sato},
\begin{align*}
  \ft\psi(t) =
  \exp[-C\mean f_\alpha(\ip\xi t) + \iunit \ip \tau t \cf{\alpha=1}],
  \quad t\in \ssp,
\end{align*}
where $C>0$ and $\tau\in\ssp$ are constants, $\xi\in S^{d-1}$ with
$\mean\xi=0$ if $\alpha=1$, and for $\theta\in \Reals$,
$f_\alpha(\theta) =|\theta|^\alpha [1-\iunit 
\tan(\pi\alpha/2)\sign(\theta) \cf{\alpha\ne  1} +
\iunit(2/\pi)\sign(\theta) \ln|\theta| \cf{\alpha=1}]$.  Conversely,
for any $C>0$, $\tau\in\ssp$, and $\xi$ on $S^{d-1}$, provided
$\mean\xi=0$ if $\alpha=1$, the RHS is the characteristic function of
a nonconstant strictly stable distribution.

\begin{prop} \label{p:radial}
  If $\xi$ has a bounded density with respect to the spherical measure
  on $S^{d-1}$, then \eqref{e:uniform-SRT2} holds for $G$, i.e., $\sup
  |\varrho_\delta - \varrho_0|\to 0$ as $\delta \to 0$.
\end{prop}

Finally, it is of interest to infer properties of an ID distribution
from its \levy measure (cf.\ \cite{\Bing}, 8.2.7;
\cite{feller:71:jws}, XVII.4; \cite {embrechts:81:ap}).  This is the
motivation of the following result.
\begin{prop} \label{p:bounded-ratio-ind}
  Let $\alpha\in (0,2)$ and $F\in\DA_0(\alpha)$ be ID with \levy
  measure $\nu$.  Define $A_\nu(s) = 1/\nu(\ssp\setminus s B_d)$ and
  $\phi_\nu(x) = |x|^d \nu(x+ h I_d) A_\nu(|x|)$.  If condition \eqref
  {e:bounded-ratio} is satisfied with $\phi$ and $A$ replaced with
  $\phi_\nu$ and $A_\nu$, respectively, then the SRT holds for $F$.
\end{prop}

\section{Outline of proof} \label{s:outline}
\subsection{Preliminaries}
Several facts about distributions in the domain of attraction will be
needed.  For random variable $X$ in $\ssp$, for $s>0$ and $u\in\ssp$,
denote
\begin{align*}
  \begin{split}
    q_X(s)=\pr{|X|>s},
    \quad&
    c_X(s)=\mean[X\cf{|X|\le s}],
    \\
    m_X(s,u)=\mean[\ip u X^2\cf{|X|\le s}],
    \quad&
    V_X(s)=\mean[|X|^2\cf{|X|\le s}].
  \end{split}
\end{align*}
For $\inum x\in\ssp$, denote $S_0(x)=0$ and $S_n(x) = S_{n-1}(x) +
x_n$, $n\ge 1$.  For $0<\alpha\le 2$, denote $F\in\DA(\alpha)$ if
there are $a_n\in\Reals$ and $b_n\in\ssp$, such that for $\inum X$
\iid $\sim F$, $S_n(X)/a_n - b_n$ weakly converges to an
$\alpha$-stable law that is nondegenerate (\cite {\Sato}, Def.~24.16).
See \cite{\Rva}, Th.~4.1--4.2, for necessary and sufficient conditions
for $F\in \DA(\alpha)$.  In particular, if $X\sim F\in \DA(\alpha)$,
then
\begin{align} \label{e:prelim-V} 
  V_X\in \RV_{2-\alpha},
\end{align}
where $\RV_\theta$ denotes the class of functions that are regularly
varying at $\infty$ with exponent $\theta$, and
\begin{align} \label{e:prelim-q-V}
  q_X(s) = [2/\alpha-1+o(1)] V_X(s) /s^2, \quad \text{as}\ s\toi.
\end{align}
Let $A$ be any function such that 
\begin{align} \label{e:prelim-A}
  A(s) \sim s^2 / V_X(s) \quad\text{ as } s\toi.
\end{align}
Then for any sequence $a_n$ such that $A(a_n)/n\to 1$ as $n\toi$,
\begin{align}  \label{e:prelim-WC}
  S_n(X)/a_n - b_n\conv D \mu \quad\text{for suitable $b_n$},
\end{align}
where $\mu$ is the aforementioned stable law.   Define $a_0=1$.  By
definition, $F\in 
\DA_0(\alpha)$ if \eqref{e:prelim-WC} holds with $b_n=0$, in which
case $\mu$ is strictly stable (\cite{breiman:92:siam}, \S9.8).  For
$F\in\DA(\alpha)$,
\begin{align} \label{e:centering}
  F\in\DA_0(\alpha)  \Iff
  \text{$(n/a_n) c_X(a_n)$ converges as $n\toi$}.
\end{align}

Proofs of the above facts are readily available for the univariate
case (\cite{\Bing}) but not so for the multivariate case.  For
convenience, their proofs are given in the Appendix \ref
{s:mv-stable}.

\subsection{Components of proof}
As noted in Section \ref{s:intro}, the probabilistic approach to the
SRT deals with the big-$n$ and small-$n$ contributions \eqref
{e:big-small-n} in different ways.  Theorem \ref{t:SRT} is a
consequence of the following results.   
\begin{theorem}[Big-$n$ contribution] \label{t:big-n}
  Let $F\in \DA_0(\alpha)$, $A$, $\psi$, and $\varrho_s$ be as in
  Theorem \ref{t:SRT}.  Let $X\sim F$ and $T\in \Reals^{d\times d}$ be
  nonsingular such that $T X = (L, W, Z)$ as in \eqref
  {e:normalization}.  Give $\delta>0$ and $h>0$, define $\Delta_h$ as
  in \eqref{e:cell} and
  \begin{align*}
    r_{\delta,h}(s\omega)
    = 
    \frac{s^d}{A(s)} \sum_{n\ge A(\delta s)}
    \Pseq n{s\omega + T^{-1} \Delta_h}.
  \end{align*}
  Then as $s\toi$,
  \begin{align} \label{e:big-n}
    \sup_{\omega\in S^{d-1}} \Abs{
      r_{\delta,h}(s\omega)
      - \frac{h^{d-\nu} \varrho_\delta(\omega)}{|\det T|}
    } = o_{\delta,h}(1).
  \end{align}
\end{theorem}

\begin{theorem}[Small-$n$ contribution] \label{t:small-n}
  Let $F\in \DA_0(\alpha)$ and $A$ be as in Theorem \ref{t:SRT}.
  Define $\kappa$ as in \eqref{e:kappa}. Then given $0 < \theta <
  1/\kappa$, $\eta>0$, and $\rx>0$, for $0<\delta \ll_{\theta, \eta,
    \rx} 1$, $s\gg_{\theta, \eta, \rx, \delta,h} 1$, and $n\le
  A(\delta s)$,
  \begin{align} \label{e:small-n-local}
    \sup_{\omega\in S^{d-1}}
    \Pseq n{s\omega + h I_d}
    \ll_h
    n a^{-d}_n\sup_{|t|>\theta s} K(t, a_n, \eta, h)
    + \rx_n(s)
  \end{align}
  with $\rx_n(s)>0$ satisfying
  \begin{align} \label{e:small-n-r}
    \sum_{n\le A(\delta s)} \sup_{r\ge s} \rx_n(r)\ll \rx A(s)/s^d.
  \end{align}
  In particular,
  \begin{align*}
    \sum_{n\le A(\delta s)}
    \sup_{|x|\ge s}
    \Pseq n{x + h I_d}
    \ll_h
    \sum_{n\le A(\delta s)}
    n a^{-d}_n\sup_{|t|>\theta s} K(t, a_n, \eta, h)
    + \rx A(s)/s^d.
  \end{align*}
\end{theorem}

It is clear that the big-$n$ contribution $r_{\delta,h}(s\omega)$
depends on the the lattice-nonlattice composition of $X$.  In
contrast, in dealing with the small-$n$ contribution, the
lattice-nonlattice composition is unimportant.  Once the two theorems
are proved, Theorem \ref{t:SRT} immediately follows from the next
result, which itself has some application; see Example \ref{ex:w-5A}.

\begin{prop} \label{p:big-n}
  Let $F\in \DA_0(\alpha)$, $A$, $\psi$, and $\varrho_s$ be as in
  Theorem \ref{t:SRT}.  Then, if the small-$n$ contribution is
  asymptotically negligible, i.e,
  \begin{gather} \label{e:small-n}
    \lim_{\delta\to 0} 
    \Lsup_{s\toi} \sup_{\omega\in S^{d-1}} \frac{s^d}{A(s)}
    \sum_{n\le A(\delta s)} 
    \Pseq n{s\omega + h I_d} = 0,
  \end{gather}
  then \eqref{e:uniform-SRT} and \eqref{e:uniform-SRT2} hold.
  Conversely, if \eqref{e:uniform-SRT} and \eqref{e:uniform-SRT2}
  hold, then \eqref{e:small-n} holds.
\end{prop}

\begin{proof}[Remark]\
  \begin{enumerate}
  \item  If $X\in\Ints^d$ is aperiodic, then Theorem \ref {t:big-n} is
    implied by \cite{\Will}, Eq.~(3.6).  When $X$ is not strongly
    aperiodic, the proof in \cite{\Will} relies on approximating $X$
    by a strongly aperiodic one; also see \cite{\Spit}, P26.1.
    However, by Proposition \ref {p:aperiod}, for some $L$ and $D$ as
    in \eqref {e:normalization}--\eqref {e:cell0} and some $\Upsilon
    \in \Ints^{d \times d}$ with $\det\Upsilon = \pm 1$, $T X = L$
    with $T= D^{-1} \Upsilon$.  Letting $\Delta_d = D^{-1} \Upsilon
    I_d$, $U(x+I_d) = U(x+T^{-1} \Delta_d)$.  Then, without
    approximation, Theorem \ref {t:big-n} leads to Eq.~(3.6) in
    \cite{\Will}.
    
  \item In \cite{\Will}, for aperiodic $X$, it is shown that
    \eqref{e:small-n} combined with \eqref{e:uniform-SRT2} implies
    \eqref {e:uniform-SRT}.  However, by Proposition \ref{p:big-n},
    \eqref{e:small-n} implies both \eqref {e:uniform-SRT} and \eqref
    {e:uniform-SRT2}.
    
  \item Proposition \ref{p:big-n} is weaker than Proposition A of
    \cite {chi:15:aap} which essentially states that for $d=1$, $|x|
    U(x+\Delta_h)/A(|x|)$ converges as $x\to \pm\infty \Iff$ \eqref
    {e:small-n} holds, and if either happens the limit must be $h
    \varrho_0$.  However, the argument for that result does not apply
    to $d>1$.  \qedhere
\end{enumerate}
\end{proof}

\begin{proof}[Proof of Proposition \ref{p:big-n}]
  Since $\psi$ is bounded and $d>\alpha$, $\sup_\omega
  \varrho_s(\omega) < \infty$ for each $s>0$ and $\varrho_s(\omega)
  \uparrow \varrho_0(\omega)$ as $s\downarrow 0$.  Clearly,
  $r_{\delta,h}(s\omega)\ge 0$ is decreasing in $\delta$.  By \eqref
  {e:big-n}, for $s\gg 1$, $\sup_\omega r_{1,h}(s\omega) < \infty$.
  Then from $r_{0,h}(s\omega) - r_{1,h}(s\omega)  =
  (s^d/A(s))\sum_{n<A(s)} \Pseq n{s\omega + \Delta_h}\le s^d$, 
  it follows that $\sup_\omega r_{0,h}(s\omega) < \infty$.  Therefore,
  the differences in \eqref{e:uniform-SRT} and \eqref{e:uniform-SRT2}
  are well-defined.

  Suppose \eqref{e:small-n} holds.  It is easy to see that \eqref
  {e:small-n} still holds if $h I_d$ is replaced with any bounded set,
  in particular $\Delta_h$.  Then given $\rx>0$, there is $\eta>0$,
  such that for any $0<\delta\le \eta$ and $s\gg_\eta 1$, $0\le
  \sup_\omega |r_{0,h}(s\omega) - r_{\delta,h}(s\omega)| \le
  \sup_\omega |r_{0,h}(s\omega) - r_{\eta,h}(s\omega)| \le h^{d-\nu}
  \rx$.  By Theorem \ref{t:big-n}, for $s\gg_\delta 1$, $\sup_\omega
  |r_{\delta,h}(s\omega)- h^{d-\nu} \varrho_\delta(\omega)| \le
  h^{d-\nu}\rx$.  Combining the inequalities, $\sup_\omega|r_{0,h}(s
  \omega) - h^{d-\nu} \varrho_\delta(\omega)| \le 2h^{d-\nu}\rx$.  As
  the inequality holds for $\eta$ and $\delta$, $\sup_\omega
  |\varrho_\delta(\omega) - \varrho_\eta(\omega)|\le 4\rx$.  Letting
  $\delta\to 0$ yields \eqref {e:uniform-SRT2} and then \eqref
  {e:uniform-SRT} follows easily.  Conversely, if \eqref
  {e:uniform-SRT} and \eqref {e:uniform-SRT2} hold, then by Theorem
  \ref{t:big-n}, given $\rx>0$, there is $\delta>0$ such that
  $\Lsup_{s\toi} \sup_\omega |r_{0,h}(s\omega) - r_{\delta,h}(s
  \omega)| < \rx$, so \eqref{e:small-n} holds if $h I_d$  therein is
  replaced with $\Delta_h$.  Since $h I_d$ can be covered by a finite
  number of $z + \Delta_h$, then \eqref{e:small-n} follows.
\end{proof}

As an application of Proposition \ref{p:big-n}, consider a classical
example on multivariate SRT given in \cite{\Will}.  The following
formulas will be used in Example \ref{ex:w-5A} and in the proofs in
Section \ref {s:proof-props}.  First, it can and will always be
assumed without loss of generality that $A$ is strictly increasing and
\begin{align} \label{e:prelim-A1}
  A(0)=0, \quad
  A'(s)\asymp A(s)/s\ \text{ for } s>0.
\end{align}
Then given $\beta$, for $s\gg 1$, by change of variable and $A'(s)
\asymp A(s)/s$,
\begin{align}  \label{e:L}
  \tilde A_\beta(s):=\sum_{n\le A(s)} n a^{-\beta}_n
  \ll \int_1^{A(s)} \frac{u\,\dd u}{(A^{-1}(u))^\beta}
  \ll \int_{a_1}^s \frac{A(u)^2\,\dd u}{u^{\beta+1}}
  \ll
  \begin{cases}
    A(s)^2 s^{-\beta} & \text{if } \alpha>\beta/2 \\
    O(1) & \text{if } \alpha<\beta/2.
  \end{cases}
\end{align}

\begin{example} \label{ex:w-5A}
  Consider the following modified version of Example 5-A in \cite
  {\Will}.  Let $d>1$.  Let $\xi\in\Ints\setminus\{0\}$, such that for
  $k\in\Nats$,
  \begin{align*}
    \pr{\xi=k} = \pr{\xi=-k} =
    \begin{cases}
      c k^{-1-d/2}\ln k & \text{if } k\not\in\{2^n: n\ge 1\}
      \\
      c k^{-d/2}/b_k & \text{otherwise}
    \end{cases}
  \end{align*}
  where $c>0$ is the normalizing constant and $b_k\gg 1$.  Let $X =
  (\eno X d)$, with $X_i$ \iid $\sim\xi$.  Then $X\in \DA_0(\alpha)$
  with $\alpha = d/2$.  The limiting stable density is $\psi(u)
  = g(u_1)\cdots g(u_d)$, with $g$ the univariate symmetric
  $\alpha$-stable density.  Then for any $\omega\in S^{d-1}$,
  since it has at least one coordinate with absolute value $\ge 1/d$,
  $\psi(s\omega) \ll \sup_{1/d\le a\le 1} g(a s) \ll s^{-\alpha-1}$,
  giving $\psi(s\omega) s^{d-\alpha-1} \ll s^{d-2\alpha - 2} = s^{-2}$
  for $s\gg 1$.  As a result, \eqref {e:uniform-SRT2} holds. 

  On the other hand, by \cite{\Will}, if $d=2,3$ and $b_k \ll 1+\ln
  k$, then the SRT fails to hold for $X$.  It will be shown next that
  if $d=2,3$, then the SRT \eqref {e:uniform-SRT} holds $\Iff \ln k =
  o(b_k)$ as $k\toi$, and if $d=4$, then the SRT  \eqref
  {e:uniform-SRT} holds $\Iff (\ln k)^2 = o(b_k)$.

  Without loss of generality, let $h=1$, so $h I_1 = [0,1)$.  Let
  $b(z)$ be a function such that $b(z)\equiv b_k$ in each $[k,k+1)$.  
  First, let $d=2$ or $3$.  It suffices to show that if $\ln k =
  o(b_k)$, then the SRT \eqref{e:uniform-SRT} holds.  By \cite{\Will},
  $X\sim -X\in\DA_0(\alpha)$ with $\alpha = d/2\in (0,2)$.  Since
  $\alpha \in (0,2)$, from
  \begin{align*}
    \pr{\xi>s} \le q_X(s) = \pr{|X|>s} \le d \pr{|\xi|>s/\sqrt d}
    \asymp \pr{\xi>s} \asymp s^{-d/2}\ln s, \quad s\gg 1,
  \end{align*}
  it follows that $A(s) \sim 1/q_X(s) \sim C s^{d/2}/\ln s$ for some
  constant $C>0$.

  Fix $\theta>0$ and $0< \eta < 1/(10d)$.  Then the bound \eqref
  {e:K-ind-comp} still holds, i.e., for $t=(\eno t d)$, letting $x =
  t_i$ with $|t_i| = \max |t_j|$,
  \begin{align*}
    K(t,a, \eta, h)
    \le 
    \int_{|u|<\eta d |x|}
    \pr{\xi\in x -u + I_1} \exf{-|u|/a}\,\dd u.
  \end{align*}
  For each $u$ with $|u|<\eta d|x|$, $x - u + I_1$ has exactly one
  $k\in\Ints$.   For $|t|\gg 1$, $|k| \asymp |x - u| \asymp |x| \asymp
  |t|$.  If $k\not\in \{\pm 2^n: n\ge 1\}$, then $\pr{\xi\in x - u +
    I_1} = c |k|^{-1-d/2}\ln |k| \asymp |t|^{-1-d/2}\ln |t| =
  q_X(|t|)/|t|$, and likewise, if $k\in \{\pm 2^n: n\ge 1\}$, then
  $\pr{\xi\in x - u + I_1} \asymp |t|^{-d/2}/b(t) \ll q_X(|t|)/(b(|t|)
  \ln |t|)$.  It can be seen that the set of $u\in (-\eta d |x|, \eta
  d |x|)$ with $x - u + I_1$ containing one $k\in\{\pm 2^n: n\ge 1\}$
  is a single interval of length at most 1.  Then
  \begin{align*}
    K(t,a, \eta, h)
    &\ll 
    \frac{q_X(|t|)}{|t|}\int \exf{-|u|/a}\,\dd u 
    + \frac{q_X(|t|)}{b(|t|)\ln |t|}
    \ll
    \frac{a q_X(|t|)}{|t|}
    + \frac{q_X(|t|)}{b(|t|)\ln |t|}.
  \end{align*}
  Since $q_X(s)/(\ln s)^2 \asymp A(s)/s^d$, if $\ln s = o(b(s))$ as
  $s\toi$, then for $s\gg 1$,
  \begin{align} \label{e:w-K}
    \sum_{n\le A(\delta s)} n a^{-d}_n
    \sup_{|t|>\theta s} K(t,a_n, \eta, h)
    \ll
    \frac{\tilde A_{d-1}(\delta s)}{s A(s)}
    +
    o(A(s)/s^d) \tilde A_d(\delta s).
  \end{align}
  Given $\delta>0$, for $s\gg_\delta 1$, by \eqref{e:L}
  $\tilde A_{d-1}(\delta s) \ll A(\delta s)^2/(\delta s)^{d-1}\ll
  \delta s/(\ln s)^2$, while
  \begin{align*}
    \tilde A_d(\delta s) 
    \ll
    \int_{a_1}^{\delta s} \frac{A(u)^2\,\dd u}{u^{d+1}}
    \ll
    \int_{a_1}^2 \frac{A(u)^2\,\dd u}{u^{d+1}} +
    \int_2^\infty \frac{\dd u}{u(\ln u)^2} < \infty.
  \end{align*}
  Then the LHS in \eqref{e:w-K} is $O(\delta(\ln s)^{-2}/A(s)) +
  o(A(s)/s^d) = O(\delta A(s)/s^d)$.  Since $\delta$ is arbitrary,
  then by Theorem \ref{t:SRT}, the SRT holds.

  Next let $d=4$.  Then $\mean|X|^2=\infty$.  However, as $s\toi$,
  $s^2 q_X(s)\asymp \ln s$,
  \begin{align*}
    V_X(s) = \int_0^s u^2 \pr{|X|\in \dd u}
    = 2 \int_0^s u [q_X(u) - q_X(s)]\,\dd u
    \asymp (\ln s)^2,
  \end{align*}
  and for $t=(\eno t 4)$,
  \begin{align*}
    m_X(s,t) = \sum_{i=1}^4 t_i^2 \mean[X_i^2 \cf{|X|\le s}] +
    \sum_{i\ne j} t_i t_j \mean[X_i X_j \cf{|X|\le s}
    = |t|^2 V_X(s)/d.
  \end{align*}
  Then by Theorem 4.1 of \cite{\Rva}, $X\sim -X\in \DA_0(2)$ and
  $A(s) \sim C s^2/(\ln s)^2$ as $s\toi$ for some constant $C>0$.  The
  bound on $K$ just above \eqref{e:w-K} still holds.  Then
  \begin{align} \label{e:w-K-2}
    \sum_{n\le A(\delta s)} n a^{-4}_n
    \sup_{|t|>\theta s} K(t,a_n, \eta, h)
    &\ll
    \frac{q_X(s)\tilde A_3(\delta s)}{s}
    +
    \frac{q_X(s) \tilde A_4(\delta s)}{b(s) \ln s}.
  \end{align}
  Given $\delta>0$, for $s\gg_\delta 1$, as $A(s) = s^2/(\ln s)^2$, by
  similar calculation as in the case $d=2$ or $3$, $\tilde A_3(\delta
  s)\ll \delta s/(\ln s)^4$ and $\tilde A_4(\delta s) = O(1)$.  Then
  the LHS of \eqref{e:w-K-2} is $s^{-4} A(s) O(\delta/\ln s+ (\ln
  s)^2/b(s))$.  Thus, if $(\ln s)^2 = o(b(s))$, then \eqref
  {e:uniform-SRT} holds.  Conversely, if $b(s)/(\ln s)^2\not\toi$,
  then by similar argument as in \cite{\Will}, \eqref{e:small-n}
  cannot hold.  Since \eqref{e:uniform-SRT2} holds, by Proposition
  \ref{p:big-n}, the SRT \eqref{e:uniform-SRT} cannot hold.  \qed
\end{example}

\section{Basic bounds} \label{s:bounds}
\subsection{\levy concentration function} 
For a random variable $X\in\ssp$, define
\begin{align*}
  Q_X(h) = \sup_{x\in\ssp} \pr{X\in x + h I_d}, \quad h\ge 0.
\end{align*}
The function is a special case of \levy concentration function of
multivariate random variables, which has been studied before (\cite
{zigel:81:tpa, miroshnikov:89:tpa}).  The purpose here is to show the
following result.
\begin{lemma} \label{l:lc}
  There is an absolute constant $c_d>0$, such that 
  \begin{align*}
    Q_X(h)
    \le
    c_d[(1/a)\vee h]^d \int_{\Norm t \le a}
    |\chf_X(t)|\,\dd t, \quad a>0.
  \end{align*}
\end{lemma}

\begin{proof}
  The argument follows the one on p.~22--26 of \cite{\Petr}.  Let $f$
  be a probability density on $\ssp$ such that $f(x) = f(-x)$ and $\ft
  f\in L^1$.  For $y\in \ssp$ and $a>0$, by applying Fourier inversion
  formula to the density of $X + a^{-1} Y$, where $Y$ has density $f$
  and is independent of $X$,
  \begin{align*}
    \int f(a x) \,\pr{X\in y+\dd x}
    &=
    \nth{(2\pi a)^d} \int \exfi{\ip t y} \ft f(t/a) \chf_X(t)\,\dd t
    \le
    \nth{(2\pi a)^d} \int |\ft f(t/a) \chf_X(t)|\,\dd t.
  \end{align*}
  On the other hand,
  \begin{align*}
    \int f(a x) \pr{X\in y+\dd x}
    &\ge
    \int_{\Norm x \le h/2} f(a x) \pr{X\in y+\dd x}
    \\
    &\ge
    \inf_{\Norm x \le a h/2} f(x) \cdot \pr{X\in y + [-h/2, h/2]^d}.
  \end{align*}
  As a result,
  \begin{align} \label{e:Q}
    Q_X(h)
    \le
    (2\pi a)^{-d} \sup_{\Norm x \le ah /2} \nth{f(x)}
    \int |\ft f(t/a) \chf_X(t)|\,\dd t.
  \end{align}
  Now for $x=(\eno x d)$, let $f(x) = \prod f_0(x_i)$, where $f_0(y) =
  3/(8\pi) [\sin(y/4)/(y/4)]^4$ for $y\in\Reals \setminus\{0\}$ and
  $f_0(0)=3/(8\pi)$.  Then for $t=(\eno t d)$, $\ft f(t) = \prod \ft
  f_0(t_i)$, where
  \begin{align*}
    \ft f_0(t_i) =
    \begin{cases}
      0 & \text{if } |t_i|\ge 1,
      \\
      2(1-|t_i|)^3 & \text{if } |t_i|\in [1/2,1]
      \\
      1-6 t_i^2 + 6 |t_i|^3 & \text{if } |t_i|\le 1/2.
    \end{cases}
  \end{align*}
  See p.~25 \cite{\Petr}.  Let $c_d=(2\pi)^{-d} \sup_{\Norm x\le 1/2}
  1/f(x)$.  Then by \eqref{e:Q}, for $a\le 1/h$,
  \begin{align*}
    Q_X(h)
    \le
    c_d(1/a)^d \int_{\Norm t \le a} |\chf_X(t)|\,\dd t.
  \end{align*}
  On the other hand, for $a\ge 1/h$
  \begin{align*}
    c_d h^d \int_{\Norm t \le a} |\chf_X(t)|\,\dd t
    \ge
    c_d h^d \int_{\Norm t \le 1/h} |\chf_X(t)|\,\dd t
    \ge
    Q_X(h),
  \end{align*}
  where the second inequality follows from the previous display.
  Combining the two displays then finishes the proof.
\end{proof}

\subsection{Local large deviation}
The following bounds will be used in the proof of Theorem
\ref{t:small-n}.
\begin{prop} \label{p:local-LDP}
  Let $F\in \DA_0(\alpha)$ and $X,\inum X$ be \iid $\sim F$.
  Put $M_0=0$ and for $n\ge 1$, put $M_n = \max\{|X_1|, \ldots,
  |X_n|\}$.   Then there are $s_0>0$, $C>0$ both only depending on $\{F,
  A\}$, such that for all $x\in\ssp$, $s\ge s_0$, $h>0$, and $n\ge 0$,
  \begin{align*}
    \pr{S_n(X)\in x + h I_d\AND M_n \le s}
    \ll_h (s^{-d} + a^{-d}_n) \exf{-|x|/s + C n/A(s)}.
  \end{align*}
\end{prop}
The bound is a multivariate generalization of the local large
deviation bounds in \cite{denisov:08:ap, \Doney, chi:14:tr}.  Letting
$s=\infty$, the bound yields $Q_{S_n(X)}(h) \ll_h a^{-d}_n$, which can
also be derived from the LLT of Stone (cf.~Proposition \ref{p:llt}).
The case $n=0$ is included in Proposition \ref{p:local-LDP} only for
convenience, which follows by noting $\pr{S_0(X)\in x+h I_d} =
\cf{-x\in h I_d} \le \cf{|x|\le h\sqrt d}$.  Let $n\ge 1$ henceforth.
The proof is based on several lemmas which will be shown later.

\begin{lemma}\label{l:s0}
  There is $s_0>0$ such that
  \begin{align*}
    \inf_{\omega\in S^{d-1}}
    \mean[\ip \omega {X_1 - X_2}^2\cf{|X_1|\vee |X_2|\le s_0/(2\sqrt
      d)}] > 0.
  \end{align*}
\end{lemma}

Fix $s_0>0$ as in Lemma \ref{l:s0}.  Then for $\omega\in S^{d-1}$
and $s\ge s_0$,
\begin{align*}
  Z(s,\omega) := \mean[\exf{\ip \omega X/s}\cf{|X|\le s}]\in (0,e).
\end{align*}
Define the following probability measure concentrated in $s B_d$,
\begin{align*}
  G_{s,\omega}(\dd x)=Z(s,\omega)^{-1} \exf{\ip \omega x/s}
  \cf{|x|\le s} F(\dd x).
\end{align*}
Let $x = r v$, where $r=|x|$ and $v\in S^{d-1}$.  Let $Y,\inum
Y$ be \iid $\sim G_{s,v}$.  Then
\begin{align*}
  \pr{S_n(X) \in x + h I_d\AND M_n\le s}
  &=
  [Z(s,v)]^n
  \mean[\exf{-\ip v{S_n(Y)}/s} \cf{S_n(Y) \in x+ h I_d}].
\end{align*}

The function $Z(s,v)$ on the RHS has the following property.
\begin{lemma} \label{l:log-phi}
  $[\ln Z(s,\omega)]_+ \ll 1/A(s)$ for $s\ge s_0$ and $\omega\in
  S^{d-1}$.
\end{lemma}
Thus, there is a constant $C = C(F,A)>0$ such  that $Z(s, \omega)
\le e^{C/A(s)}$ for $s\ge s_0$ and $\omega\in S^{d-1}$.  On the
other hand, since $\ip v y - \ip v x \ge - |y-x| \ge -\sqrt d
h$ for $y\in x+h I_d$, for $s\ge s_0$,
\begin{align*}
  \exf{-\ip v{S_n(Y)}/s}\cf{S_n(Y) \in x +  h I_d}
  &\ll_h
  \exf{-\ip v x/s}\cf{S_n(Y)\in x+h I_d}
  \\
  &=
  \exf{-r/s}  \cf{S_n(Y) \in x + h I_d}.
\end{align*}
As a result,
\begin{align}
  \pr{S_n(X) \in x + h I_d\AND M_n\le s}
  &\le
  [Z(s,v)]^n \exf{-r/s} \pr{S_n(Y)\in x + h I_d}
  \nonumber \\
  &\le
  [Z(s,v)]^n \exf{-r/s} Q_{S_n(Y)}(h)
  \ll
  \exf{-r/s + C n/A(s)} Q_{S_n(Y)}(h).
  \label{e:Q1}
\end{align}
By Lemma \ref{l:lc}, if $W = Y_1 - Y_2$, then $\chf_W(t) =
|\chf_Y(t)|^2 >0$ and
\begin{align}
  Q_{S_n(Y)}(h)
  &\le
  c_d(s_0\vee h)^d \int_{\Norm t\le 1/s_0}
  |\chf_{S_n(Y)}(t)|\,\dd t
  \nonumber \\
  &\ll_h
  \int_{\Norm t\le 1/s_0} \chf_W(t)^{n/2}\,\dd t
  \ll_h
  s^{-d} +
  \int_{1/s<|t|\le \sqrt d/s_0} \chf_W(t)^{n/2}\,\dd t.
  \label{e:Q2}
\end{align}
By $x\le\exf{-(1-x)}$, $\chf_W(t) = \mean\cos\ip t W \le
\exf{-(1-\mean\cos\ip t W)/2}$.  By $1-\cos x\gg x^2$
for $|x|\le 1$,
\begin{align*}
  1-\mean\cos\ip t W
  &
  \ge
  \mean[(1-\cos \ip t W)\cf{|\ip t W|\le 1}]
  \\
  &\gg
  \mean[\ip t {Y_1 - Y_2}^2\cf{|\ip t{Y_i}|\le 1/2\AND i=1,2}]
  \\
  &=
  Z(s,v)^{-2} \mean[\ip t {X_1 - X_2}^2 \exf{\ip v {X_1 +
      X_2}/s}\cf{|\ip t{X_i}|\le 1/2\AND |X_i|\le s\AND i=1,2}].
\end{align*}
Given $t\in\ssp$ with $1/s<|t|\le \sqrt d/s_0$, let
$\omega=t/|t|$ and $L=|t|^{-1}/2$.  Then
\begin{align} \label{e:W}
  1-\mean\cos\ip t W
  &\gg
  |t|^2
  \mean[\ip \omega {X_1 - X_2}^2 \cf{|X_i|\le L\AND i=1,2}].
\end{align}
Since $X_1$ and $X_2$ are \iid $\sim X$,
\begin{align*}
  \mean[\ip \omega {X_1 - X_2}^2 \cf{|X_i|\le L}]
  =
  2 m_X(L,\omega) \pr{|X|\le L} -
  2 |\ip \omega {c_X(L)}|^2.
\end{align*}
Since $L\ge s_0/(2\sqrt d)$, by Lemma \ref{l:s0}, the infimum of the
LHS over $\omega\in S^{d-1}$ is positive.  In particular
$m_X(L,\omega)\ge \inf_{\omega\in S^{d-1}\AND s\ge s_0/(2\sqrt d)}
m_X(s,\omega)>0$.  Assume the following is true for now.

\begin{lemma} \label{l:2nd-moment}
  $m_X(s,\omega)\asymp s^2/A(s)$ for $s\ge s_0$ and $\omega\in
  S^{d-1}$.
\end{lemma}
It follows that $m_X(L,\omega) \pr{|X|\le L} \asymp |t|^{-2}/A(1/|t|)$
for $|t|\le \sqrt d/s_0$ with $L=|t|^{-1}/2$.  On the other hand, by 
\eqref{e:centering}, $|\ip \omega {c_X(L)}|^2 \le |c_X(L)|^2 \ll
|t|^{-2}/ A(1/|t|)^2$.  This combined with last two displays yields
that, for $1/s<|t|\le \sqrt d/s_0$, $1-\mean\cos\ip t W \gg 1/
A(1/|t|)$, and so for some $c>0$ that only depends on $\{F, A\}$,
$\chf_W(t)\le \exp\{-2c/A(1/|t|)\}$.   Then
\begin{align*}
  \int_{1/s<|t|\le \sqrt d/s_0} \chf_W(t)^{n/2}\,\dd t
  &\le
  \int_{1/s<|t|\le \sqrt d/s_0} \exp\Cbr{-\frac{c n}{A(1/|t|)}}\,\dd t
  \\
  &\ll
  \int_{1/s}^{\sqrt d/s_0}
  y^{d-1} \exp\Cbr{-\frac{c A(a_n)}{A(1/y)}}\,\dd y.
\end{align*}
By Potter's Theorem (\cite{\Bing}, Th.~1.5.6), $A(a_n)/A(1/y)
\gg (a_n y)^{\alpha/2}\wedge (a_n y)^{3\alpha/2}$ for $n\ge 1$ and
$y\le \sqrt d/s_0$.  Combining this with the above display and
$\exf{-(x\wedge y)} \le \exf{-x} + \exf{-y}$, there is $C=C(F,A)>0$
such that
\begin{align*}
  \int_{1/s<|t|\le \sqrt d/s_0} \chf_W(t)^{n/2}\,\dd t
  &\ll
  \int_{1/s}^{\sqrt d/s_0}
  y^{d-1} (\exf{-C(a_n y)^{\alpha/2}} + \exf{-C(a_n
    y)^{3\alpha/2}})\,\dd y.
\end{align*}
On the other hand, for any $b>0$ and $q>0$,
\begin{align*}
  \int_{1/s}^\infty y^{d-1} \exf{-b (a_n y)^q}\,\dd y
  = a_n^{-d} \int_{a_n/s}^\infty y^{d-1} \exf{-b y^q}\,\dd y
  \ll_{b,q} a_n^{-d}.
\end{align*}
The above three displays combined with \eqref{e:Q1} and \eqref{e:Q2}
then prove Proposition \ref{p:local-LDP}.

\begin{proof}[Proof of Lemma \ref{l:s0}]
  Put $\mu(\omega, s) = \mean[\ip \omega{X_1 - X_2}^2\cf{|X_1|\vee
    |X_2| \le s/2}]$.  Since $F\in \DA_0(\alpha)$ is
  nondegenerate, for each $\omega\in S^{d-1}$, $\mu(\omega,
  \infty)>0$, so by monotone convergence, there is $s(\omega)>0$ with
  $\mu(\omega, s(\omega))>0$.  Fixing $\omega$, by continuity of the
  mapping $v\to \mu(v, s(\omega))$, there is $r(\omega)>0$, such that
  $\mu(v, s(\omega))\ge \mu(\omega, s(\omega))/2$ for $v\in [\omega +
  r(\omega) B_d]\cap S^{d-1}$.  Since $S^{d-1}$ is compact, there are
  a finite number of $\omega_i\in S^{d-1}$, such that $S^{d-1}$ is
  covered by the union of $\omega_i + r(\omega_i) B_d$.  Then it is
  easy to see that $s_0 = \max_i s(\omega_i)$ has the asserted property.
\end{proof}

\begin{proof}[Proof of Lemma \ref{l:log-phi}]
  By $Z(s,\omega)=\mean[\exf{\ip \omega X/s}\cf{|X|\le s}]$
  and $\ln x \le x-1$ for $x>0$,
  \begin{align}
    \ln Z(s,\omega)
    &\le
    \mean[\exf{\ip \omega X/s}\cf{|X|\le s}]-1
    \nonumber\\
    &\le
    \mean[(\exf{\ip \omega X/s}-1)\cf{|X|\le s}]
    =
    I(s,\omega) + \ip \omega {s^{-1} c_X(s)},
    \label{e:phi-bound}
  \end{align}
  where $I(s,\omega) = \mean[(\exf{\ip \omega X/s}-1 - \ip \omega
  X/s)\cf{|X|\le s}]$.  By $|\exf z - 1 - z|\le c z^2$ for $|z|\le 1$,
  where $c>0$ is an  absolute constant, for $s\ge s_0$, $|I(s,
  \omega)| \ll s^{-2} m_X(s,\omega) \le s^{-2} V_X(s) \sim
  1/A(s)$.  On the other hand, from \eqref{e:centering},
  $\sup_\omega |\ip \omega {s^{-1} c_X(s)}| \ll 1/A(s)$.
  By \eqref{e:phi-bound}, the proof is complete.
\end{proof}

\begin{proof}[Proof of Lemma \ref{l:2nd-moment}]
  Since for all $s\ge s_0$ and $\omega\in S^{d-1}$, $0<m_X(s,
  \omega)\le V_X(s) \asymp s^2/A(s)$, it suffices to show that for
  $s\ge s_0$, $\inf_{\omega\in S^{d-1}} m_X(s,\omega)\gg s^2/A(s)$.
  For $u$, $v\in \ssp$, by $|\ip u X^2 - \ip v X^2| = |\ip {u - v} X
  \ip {u + v} X|\le |u-v| | u + v| |X|^2$,
  \begin{align*}
    |m_X(s,u) - m_X(s,v)| \le |u - v| |u + v| V_X(s).
  \end{align*}
  In particular, for $u, v\in S^{d-1}$, $|m_X(s, u) - m_X(s,v)|\le 2
  |u - v| V_X(s)$.  Then, by the compactness of $S^{d-1}$, it suffices
  to show that given $\omega\in S^{d-1}$, for $s\ge s_0$, $m_X(s,
  \omega)\gg s^2/A(s)$.  First, let $\alpha=2$.  Let $\Sigma$ be the
  covariance matrix of the limit normal distribution and put $b(u) =
  \ip u {\Sigma u}$.  By \cite{\Rva}, Th.~4.1, $m_X(s, u)/m_X(s, v)\to
  b(u)/b(v)$.  Letting $v = e_i$, it follows that
  \begin{align} \label{e:m-V-normal}
    m_X(s,u) / V_X(s) = m_X(s, u)/\sum_i m_X(s, e_i)\to b(u)/\sum_i
    b(e_i),
  \end{align}
  which together with
  \eqref{e:prelim-A} leads to the desired result.  Now let $\alpha\in
  (0,2)$.  Then
  \begin{align*}
    m_X(s,\omega)
    &=
    \int_{z\in [0,x]\AND v\in S^{d-1}} z^2\ip \omega v^2
    \pr{|X|\in \dd z\AND X/|X|\in\dd v}
    \\
    &=
    2\int_{z\in [0,s]\AND v\in S^{d-1}} \ip \omega v^2
    \Grp{\int_0^z x\,\dd x}
    \pr{|X|\in \dd z\AND X/|X|\in\dd v}
    \\
    &=
    2\int_{x\in [0,s]\AND v\in S^{d-1}}
    \ip \omega v^2 x \pr{x\le |X|\le s\AND X/|X|\in\dd v}\,\dd x
  \end{align*}
  Let $s\toi$.  By Th.~4.2 of \cite{\Rva} and Th.~14.10 of \cite
  {\Sato}, there is a finite nonzero measure $\gamma$ on $S^{d-1}$,
  such that for any measurable $E\subset S^{d-1}$, $\pr{|X|\ge s\AND
    X/|X| \in E} / q_X(s) \to \gamma(E)/\gamma(S^{d-1})$.  Then
  standard argument based on Riemann sum approximation to the integral
  over $v\in S^{d-1}$ yields
  \begin{align*}
    m_X(s,\omega)
    &=2\Sbr{
      \int \ip \omega v^2 \gamma(\dd v) + o(1)
    }
    \int_0^s x [q_X(x) - q_X(s)]\,\dd x
    \\
    &=
    \frac{[c+o(1)]s^2}{(2-\alpha) A(s)} \Sbr{
      \int \ip \omega v^2 \gamma(\dd v) + o(1)
    },
  \end{align*}
  where $c=c(F,A)>0$ is a constant.  Since the limiting stable law of
  $S_n(X)/a_n$ is nondegenerate, by Lemma 3.1 of \cite{\Rva}, $\int
  \ip \omega v^2 \gamma(\dd v)>0$.  Then the proof is complete.
\end{proof}

\section{Big-\emph{n} contribution}  \label{s:big-n}
This section proves Theorem \ref{t:big-n}.  The following LLT will be
used.
\begin{prop}[Stone \cite{\Stone}] \label{p:llt}
  Let $Y\in\Reals^r$ and $Z \in \Reals^{d-r}$ and
  constant vector $\beta\in\Reals^r$ be as in Proposition \ref
  {p:lattice-decomp}(3).  Let $\xi_i=(Y_i, Z_i)$, $i\ge 1$, be \iid
  $\sim (Y,Z)$.  Let $a_n\toi$ and $d_n = (b_n, c_n)\in \Reals^r
  \times \Reals^{d-r}$ such that $S_n(\xi)/a_n - d_n$ weakly converges
  to a stable law with density $\psi(y,z)$.  Denote $\Lambda_n =
  (n\beta + \Ints^r)\times\Reals^{d-r}$.  Then as $n\toi$,
  \begin{align*}
    \sup_{(y,z)\in \Lambda_n\AND h\le 1}
    \Abs{
      a^d_n \pr{S_n(Y) = y\AND S_n(Z) \in z+h I_{d-r}}
      - h^{d-r} \psi\Grp{\frac{y}{a_n}-b_n, \frac{z}{a_n}-c_n}
    }\to 0.
  \end{align*}
\end{prop}

The proof consists of two steps.  First, Theorem \ref {t:big-n} is
proved assuming that no transform of $X$ is needed to reveal its
lattice-nonlattice composition.  That is, in \eqref {e:normalization},
one can let $T = \Id_d$ so that
\begin{align} \label{e:std-normal}
  X = (L, W, Z).
\end{align}
Then the general case is proved with some uniform convergence
argument.
\begin{proof}[Proof of Theorem \ref{t:big-n} under
  \eqref{e:std-normal}]
  In the following, $x$ and $s\omega$ with $s\ge 0$ and $\omega \in
  S^{d-1}$ will be used interchangeably.  Writing $x = (u,z)$ with
  $u\in \Reals^\nu$, for any $k\in\Nats$, $\Pseq n{x+\Delta_h}$ is
  equal to the sum of $\Pseq 
  n{x_a + \Delta_{h/k}}$ over $a\in \{0, 1, \ldots, k-1\}^{d-\nu}$,
  where $x_a = (u, z+ha/k)$.  Thus, if \eqref{e:big-n} holds for $h
  \in (0,1)$, then it holds for all $h>0$.  So without loss of
  generality, let $h\in (0,1)$.

  Fix an arbitrary $M>\delta\vee 1$.  Put $J_{s,\delta,M} = [A(\delta
  s)\AND A(Ms)) \cap\Nats$ and
  \begin{align*}
    B_{\delta,M}(x,h)
    =
    \sum_{n\in J_{s,\delta,M}} \Pseq n{x+\Delta_h}.
  \end{align*}
  Then $r_{\delta,h}(x)=s^d B_{\delta,M}(x,h)/A(s) + r_{M,h}(x)$.
  First, by Proposition \ref{p:llt}, $\Pseq n{x+\Delta_h}\ll
  a^{-d}_n$.  Then by $A^{-1}  \in\RV_{1/\alpha}$ and $A^{-1}(t) =
  [1+o(1)] a_n$ for $t\in [n-1,n+1]$ as
  $n\toi$,
  \begin{align*}
    r_{M,h}(x)
    &\ll
    \frac{s^d}{A(s)}\sum_{n\ge A(Ms)} a^{-d}_n
    \ll
     \frac{s^d}{A(s)}\int_{A(Ms)}^\infty \frac{\dd t}{(A^{-1}(t))^d}.
  \end{align*}
  By change of variable $t = A(s u)$ and $A'(x)= [\alpha+o(1)] A(x)/x$
  as $x\toi$,
  \begin{align*}
    \int_{A(Ms)}^\infty \frac{\dd t}{(A^{-1}(t))^d}
    \ll
    \int_M^\infty \frac{A(s u)\,\dd u}{s^d u^{d+1}}
    \ll
    \frac{A(s)}{s^d} \int_M^\infty u^{\alpha-d-1}\,\dd u.
  \end{align*}
  Since $d>\alpha$, the above two displays show that $r_{M,h}(x)\ll
  \int_M^\infty u^{\alpha-d-1}\,\dd u$ is arbitrarily small if $M$ is
  large enough.  Also, since $\psi$ is bounded, $\sup_{\omega\in
    S^{d-1}} \int_0^{1/M} \psi(u\omega) u^{d-\alpha-1} \,\dd u$ is
  arbitrarily small as well.  Therefore, to show \eqref{e:big-n}, it
  only remains to show
  \begin{align} \label{e:big-n2}
    \frac{s^d B_{\delta,M}(x,h)}{A(s)}
    = 
    \alpha q^{-1} h^{d-\nu} 
    \int_{1/M}^{1/\delta}\psi(u\omega) u^{d-\alpha-1}\,\dd u +
    o_{\delta,M,h}(1).
  \end{align}

  Let $x = (u,w,z)\in \Reals^\nu\times \Reals^{r-\nu}\times
  \Reals^{d-r}$.  Put $\Lambda = D^{-1} \Upsilon I_\nu$.  Then
  \begin{align*}
    \{S_n(X) \in x+\Delta_h\}
    =
    \{S_n(L)\in u + \Lambda\AND S_n(W)\in w  + h I_{r-\nu}\AND
    S_n(Z)\in z + h I_{d-r}\}.
  \end{align*}
  Let $l_0 = (0, \ldots, 0, \beta_\nu) = (0, \ldots, 0, p/q)$ and $w_0
  = (\beta_{\nu+1}, \ldots, \beta_r)$ be as in \eqref {e:lattice-l-w}.
  As $L\in l_0 + \Ints^\nu$, $S_n(L)\in u + \Lambda$ only if $(u +
  \Lambda) \cap (n l_0 + \Ints^\nu) \ne\emptyset$.  Clearly, $n l_0 +
  \Ints^\nu \subset \Ints l_0 + \Ints^\nu$.  Since $p$ and $q$ are
  coprime, $D(\Ints l_0 + \Ints^\nu) = \Ints^\nu$.  Then by
  $\Upsilon^{-1} \Ints^\nu = \Ints^\nu$,
  \begin{align*}
    (u+\Lambda) \cap (\Ints 
    l_0 + \Ints^\nu) = D^{-1}[(D u + \Upsilon I_\nu) \cap \Ints^\nu]
    = D^{-1} \Upsilon[(\Upsilon^{-1} D u + I_\nu) \cap \Ints^\nu]
  \end{align*}
  has exactly one element $l$ and there is a unique number among $0,
  \ldots, q-1$, denoted $\kappa(u)$, such that $l \in \kappa(u) l_0 +
  \Ints^\nu$.  Then $l\in n l_0 +  \Ints^\nu \Iff [n- \kappa(u)]
  l_0\in \Ints^\nu \Iff q\gv n-\kappa(u)$.  It follows that
  \begin{align} \label{e:LCD}
    (u+\Lambda) \cap (n l_0  + \Ints^\nu)\ne \emptyset
    \Iff
    l\in n l_0 + \Ints^\nu
    \Iff
    n\in \kappa(u) + q\Ints.
  \end{align}
  Thus $S_n(L)\in u  + \Lambda$ only if $n\in \kappa(u) + q\Ints$.
  Next, since $W \in w_0 + \Ints^{r-\nu}$, $S_n(W)\in w + h I_{r-\nu}$ 
  only if $(w + h I_{r-\nu}) \cap (n w_0 +\Ints^{r-\nu})\ne\emptyset$.
  As a result, $\Pseq n{x+\Delta_h}>0$ only if $n\in R_x$, where
  \begin{align*}
    R_x
    =
    \{n\in\Nats: n\in\kappa(u) + q\Ints\AND (w +
    h I_{r-\nu})\cap (n w_0 + \Ints^{r - \nu})\ne \emptyset\}.
  \end{align*}  
  Then by Proposition  \ref{p:llt}, as $s\toi$, for $n\ge A(\delta s)$, 
  \begin{align*}
    \Pseq n{x+\Delta_h}
    &=
    \sum_{\substack{y=(\tilde u, \tilde w):\,
        \tilde u\in (u+\Lambda) \cap (n l_0 + \Ints^\nu)
        \\
        \tilde w \in (w + h I_{r-\nu})\cap (nw_0 + \Ints^{r-\nu})
      }
    }
    \pr{S_n(Y) = y\AND S_n(Z)\in z + h I_{d-r}}
    \\
    &=
    a^{-d}_n h^{d-r} [\psi(x/a_n) + o_\delta(1)] \cf{n\in R_x},
  \end{align*}
  where the second line is due to the fact that as $h<1$, $w + h I_{r-
    \nu}$ contains at most one point in $n w_0 + \Ints^{r-\nu}$.  Let
  $m_0$ and $m_1$  be the first and last integers in $[A(\delta 
  s), A(Ms)+1)$.  For $s\gg_{\delta,M} 1$, $m_0\asymp_\delta A(s)$
  and $m_1-m_0 \asymp_{\delta,M} A(s)$.  Fix integers $m_0=N_1 < N_2 <
  \ldots < N_k < N_{k+1} = m_1$ with $k=k(s)$, such that as $s\toi$,
  \begin{align*}
    \min_i (N_{i+1}-N_i)\toi,
    \quad
    \min_i
    (N_{i+1}-N_i)\asymp \max_i (N_{i+1} -N_i) = o_{\delta,M}(A(s)).
  \end{align*}
  Then by $A^{-1}\in \RV_{1/\alpha}$, $\max_i
  (a_{N_{i+1}} - a_{N_i}) = o_{\delta,M}(s)$.  It follows that
  uniformly for $N_i\le n< N_{i+1}$,
  \begin{align} \label{e:a-approx}
    a_n = a_{N_i} (1+o_{\delta,M}(1))
  \end{align}
  and by the continuity of $\psi$ and $s/a_n = O_\delta(1)$,
  \begin{align} \label{e:a-approx2}
    \psi(x/a_n) =\psi(x/a_{N_i}) + o_{\delta,M}(1).
  \end{align}
  Combining the above displays, by $\psi$ being bounded,
  \begin{align*}
    \Pseq n{x+\Delta_h}
    =
    h^{d-r} a^{-d}_{N_i} \psi(x/a_{N_i})
    \cf{n\in R_x} + o_{\delta,M,h}(a_n^{-d}).
  \end{align*}
  Let $g_i(x)$ be the cardinality of $R_x \cap \{N_i, N_i+1, \ldots,
  N_{i+1}-1\}$.  Then  
  \begin{align} 
    B_{\delta,M}(x,h)
    &=
    \sum_{i=1}^N \sum_{N_i\le n<N_{i+1}} F^{*n}(x+\Delta_n)
    \nonumber \\\label{e:partial-sum}
    &=
    h^{d-r} \sum_{i=1}^k a_{N_i}^{-d} \psi(x/a_{N_i})
    g_i(x)
    + o_{\delta,M,h}(1) \sum_{n\in J_{x,\delta,M}}
    a^{-d}_n.
  \end{align}
  
  Since $0<h<1$, for $1\le i\le k$ and $N_i\le n<N_{i+1}$, 
  \begin{align*}
    (w+h I_{r-\nu})\cap (n w_0 + \Ints^{r-\nu})\ne\emptyset\Iff
    \exf{2\pi\iunit (n\beta_{\nu+j} - w_j)} \in \Gamma
    \text{ for every } 1\le j\le r-\nu,
  \end{align*}
  where $\Gamma$ is the arc $\{\exf{2\pi\iunit z}: 0\le z< h\}$ of the
  unit circle $S^1$.  Let
  \begin{align*}
    b_i = \lceil (N_i - \kappa(u))/q\rceil, 
    \quad
    c_i = \lceil (N_{i+1} - \kappa(u))/q\rceil -1.
  \end{align*}
  For $s\gg_{\delta,M} 1$, $N_i - \kappa(u) > A(\delta s) - q > 0$,
  so $b_i\ge 1$.  Therefore, by \eqref{e:LCD}, for $N_i\le n <
  N_{i+1}$, $n\in \kappa(u) + q\Ints\Iff n=\kappa(u) + q k$ with
  $b_i\le k\le c_i$.  Then
  \begin{align} 
    g_i(x)
    &=
    \sum_{k=0}^{c_i - b_i} \prod_{j=1}^{r-\nu}
    \cf{
      \exf{2\pi \iunit((\kappa(u) + (b_i+k)q) \beta_{\nu+j} - w_j)}
      \in\Gamma
    }
    =
    \sum_{k=0}^{c_i - b_i} \prod_{j=1}^{r-\nu}
    \cf{
      \theta_j \exf{2\pi \iunit k \tau_j}
      \in\Gamma
    }, \label{e:ergodic-grid}
  \end{align}
  where $\theta_j = \exf{2\pi\iunit((\kappa(u) + b_i q)\beta_{\nu+j} -
    w_j)}$ and $\tau_j = q\beta_{\nu+j}$.  Then $\theta :=
  (\eno\theta{r-\nu})\in \Bb K:=(S^1)^{r-\nu}$.
  Define
  \begin{align*}
    H z = (z_1 \exf{2\pi\iunit\tau_1}, \ldots, z_{r-\nu} \exf{2\pi
      \iunit \tau_{r-\nu}}), \quad 
    z\in\Coms^{r-\nu}.
  \end{align*}
  Under the Euclidean norm, $H$ is an isometry and $H\Bb K = \Bb K$.
  Since $\eno\tau {r-\nu}$ are rationally independent, for any
  $\theta\in \Bb K$, $\{H^n\theta\}_{n\ge 0}$ is dense in $\Bb K$
  (\cite {petersen:83:cup}, p.~158).  Then the pair   $(\Bb K,H)$ is
  strictly ergodic (\cite{petersen:83:cup}, Prop.~4.2.15), and the
  normalized Lebesgue measure on $\Bb K$ is the unique probability
  measure invariant under $H$.  By Prop.~4.2.8 of \cite
  {petersen:83:cup} followed by dominated convergence, as $n\toi$,
  \begin{align*}
    \nth n \sum_{k=0}^n 
    \prod_{j=1}^{r-\nu}
    \cf{
      \theta_j \exf{2\pi \iunit k \tau_j} \in\Gamma
    }
    =
    \nth n \sum_{k=0}^n 
    \cf{H^k\theta \in\Gamma^{r-\nu}} \to h^{r-\nu}
    \quad  \text{uniformly in $\theta\in \Bb K$.} 
  \end{align*}
 Then by \eqref{e:ergodic-grid}, as
  $s\toi$, for $1\le i\le s$,
  \begin{align*}
    g_i(x) = [1+o(1)](c_i-b_i) h^{r-\nu}
    = [1+o_{\delta,M}(1)]q^{-1} (N_{i+1} - N_i) h^{r-\nu},
  \end{align*}
  which together with \eqref{e:partial-sum} and another application of
  \eqref{e:a-approx} and \eqref{e:a-approx2} yields
  \begin{align*}
    B_{\delta,M}(x,h)
    &=
    [1+o_{\delta,M}(1)] q^{-1} h^{d-\nu}
    \sum_{i=1}^k a^{-d}_{N_i} \psi(x/a_{N_i}) (N_{i+1} - N_i)
    + o_{\delta,M,h}(1) \sum_{n\in J_{x,\delta,M}} a^{-d}_n
    \\
    &=
    q^{-1} h^{d-\nu} \sum_{n\in J_{x,\delta,M}} a^{-d}_n
    \psi(x/a_n)
    + o_{\delta,M,h}(1)\sum_{n\in J_{x,\delta,M}} a^{-d}_n.
  \end{align*}
  Since $[A^{-1}(t)]^{-d} \psi(x/A^{-1}(t)) = a^{-d}_n [\psi(x/a_n) +
  o_\delta(1)]$ for $n\ge A(\delta s)$ and $t\in [n-1,n+1]$, then
  \begin{align*}
    B_{\delta,M}(x,h)
    &=
    q^{-1} h^{d-\nu} \int_{A(\delta s)}^{A(Ms)}
    \frac{\psi(x/A^{-1}(t))}{(A^{-1}(t))^d}\,\dd t
    + o_{\delta,M,h}(1)\sum_{n\in J_{x,\delta,M}} a^{-d}_n.
  \end{align*}
  As $s\toi$, by change of variable $t=A(s/u)$, the integral on the
  RHS is
  \begin{align*}
    \int_{1/M}^{1/\delta}
    \frac{\psi(u\omega)}{(s/u)^d}
    \frac{s A'(s/u)}{u^2} \,\dd u
    &=
    [1+o_{\delta,M,h}(1)] \alpha
    \int_{1/M}^{1/\delta}
    \frac{\psi(u\omega)}{(s/u)^d}\frac{A(s/u)}{u}\,\dd u
    \\
    &=
    [1+o_{\delta,M,h}(1)] \frac{\alpha A(s)}{s^d}
    \int_{1/M}^{1/\delta} 
    \psi(u\omega) u^{d-\alpha-1}\,\dd u.
  \end{align*}
  Then \eqref{e:big-n2} follows by noting that
  \begin{align*}
    \sum_{n\in J_{x,\delta}} a^{-d}_n
    &=
    [1+o_{\delta,M}(1)]
    \int_{A(\delta s)}^{A(Ms)} \frac{\dd t}{(A^{-1}(t))^d}
    =
    [1+o_{\delta,M}(1)]\frac{\alpha A(s)}{s^d}
    \int_{1/M}^{1/\delta} u^{d-1-\alpha}\,\dd u.
    \qedhere
  \end{align*}
\end{proof}

For the general case, the following corollary will be used.
\begin{cor} \label{c:big-n2}
  Under the condition \eqref{e:std-normal}, given $\delta>0$ and
  $h>0$, as $s\toi$
  \begin{align} \label{e:big-n3}
    \sup_{\omega\in S^{d-1},\, \eta\ge \delta}
    |r_{\eta,h}(s\omega) - h^{d-\nu} \varrho_\eta(\omega)|
    =
    o_{\delta,h}(1).
  \end{align}
\end{cor}
\begin{proof}
  From the preceding proof, it suffices to show
  that given $M>\delta$, as $s\toi$,
  \begin{align} \label{e:big-n4}
    \sup_{\omega\in S^{d-1},\, \delta\le\eta\le M}
    |r_{\eta,h}(s\omega) - h^{d-\nu} \varrho_\eta(\omega)|
    =
    o_{\delta,h}(1).
  \end{align}
  For $\delta\le \eta < \theta\le M$, $0\le\varrho_\eta(\omega) -
  \varrho_\theta(\omega)\ll \int^{1/\eta}_{1/\theta} u^{d-\alpha-1}
  \,\dd u\ll \eta^{\alpha-d} - \theta^{\alpha-d}$.  By $F^{*n}(x +
  \Delta_h)\ll a^{-d}_n$ and the bound at the end of the preceding
  proof, for $s\gg_\delta 1$,
  \begin{align*}
    0\le 
    r_{\eta,h}(s\omega) - r_{\theta,h}(s\omega)
    &\ll
    \frac{s^d}{A(s)}\sum_{A(\eta s)\le n \le A(\theta s)} a^{-d}_n
    \ll_\delta \eta^{\alpha-d} - \theta^{\alpha-d}.
  \end{align*}
  Thus \eqref{e:big-n4} holds if $\sup_{\omega \in S^{d-1},\eta\in E}
  |r_{\eta,h}(s\omega) - h^{d-\nu} \varrho_\eta(\omega)| = o_{\delta,
    h}(1)$, where $E$ is a finite set in $[\delta,M]$ with its
  adjacent elements being arbitrarily close.  Then Theorem \ref
  {t:big-n} under the additional condition \eqref{e:std-normal} can be
  invoked to finish the proof. 
\end{proof}

\begin{proof}[Proof of Theorem \ref{t:big-n}, general case]
  Suppose $TX = (L, W, Z)$.  Put $\tilde X = TX$.  Then $S_n(\tilde
  X)/a_n$ weakly converges to a stable law with density $\tilde
  \psi$.  For $x\ne 0$, put $s = |Tx|$ and $\omega = x/s$.  In general
  $\omega\not\in S^{d-1}$ and $|\omega|$ is a variable in 
  $x$.  However, $T\omega\in S^{d-1}$ and since $T$ is nonsingular,
  $|\omega|\ge \eta$ for some constant $\eta>0$.  Then $T^{-1} x =
  (T\omega) s$ and by $S_n(X) = T^{-1} S_n(\tilde X)$,
  \begin{align*}
    r_{\delta,h}(x)
    =
    \frac{|\omega|^d s^d}{A(|\omega|s)}
    \sum_{n\ge A(\delta|\omega| s)}
    \pr{S_n(\tilde X) \in (T\omega) s + \Delta_h}.
  \end{align*}
  Applying Corollary \ref{c:big-n2} to $S_n(\tilde X)$, as $s\toi$,
  the RHS is
  \begin{align*}
    \frac{|\omega|^d s^d}{A(|\omega|s)}
    \frac{A(s)}{s^d}
    \Sbr{h^{d-\nu} \alpha q^{-1} \int^{1/(\delta |\omega|)}_0
      \tilde\psi(u T\omega) u^{d-\alpha-1}\,\dd u + o_{\delta,h}(1)
    }.
  \end{align*}
  By change of variable $u = v/|\omega|$ and $\tilde\psi(x) =
  |\det T|^{-1} \psi(T^{-1}x)$, the above quantity is equal to
  $[1+o(1)] h^{d-\nu} \varrho_\delta(\omega/|\omega|) +
  o_{\delta,h}(1)$.  The proof is complete 
  by noting $\omega/|\omega| = x/|x|$.
\end{proof}

\section{Small-\emph{n} contribution} \label{s:small-n}
This section proves Theorem \ref{t:small-n}.  First some notation.
For $n\ge k\ge 1$, denote by $x_{n:1}$, \ldots, $x_{n:n}$ a
permutation of $\eno x n$ such that $|x_{n:i}|$ are sorted in
decreasing order and $S_{n:k}(x) = x_{n:1} + \cdots + x_{n:k}$.
Define $|x_{n:0}|=\infty$, $S_{n:0}(x)=0$, and $x_{n:k}=0$ for $k>n$.
Recall that according to \eqref{e:kappa}, $\kappa = \Flr{d/\alpha}$.

The proof follows from four lemmas.  For the first two, fix $\gamma\in
(d\alpha^{-1}(\kappa+1)^{-1},1)$, which is nonempty as $\kappa + 1>
d/\alpha$.  Define
\begin{align*}
  \zeta_{n,s} = a_n^{1-\gamma} s^\gamma, \quad n\ge 1, \ s>0.
\end{align*}

\begin{lemma} \label{l:small-n-many-large}
  Let $k=\kappa+1$ and $b = [d/(k\gamma) + \alpha]/2$.
  Note that $k\gamma b > d$ and $b\in (0,\alpha)$.  Given
  $\delta\in (0,1)$, for $s\gg_\delta 1$,
  \begin{align*}
    \sum_{n\le A(\delta s)} 
    \sup_{|x|\ge s}
    \pr{S_n(X)\in x + h I_d\AND |X_{n:k}|> \zeta_{n,|x|}}
    \ll_h
    \delta^{k\gamma b + \alpha - d} \frac{A(s)}{s^d}.
  \end{align*}
\end{lemma}

\begin{lemma} \label{l:small-jump}
  Fix $k\ge 0$ and $\rx,\delta \in (0,1)$.  For $n\ge 1$ and
  $x\in\ssp$, denote
  \begin{align*}
    E_{n,x} = \{S_n(X)\in x + h I_d\AND |X_{n:k}|>\zeta_{n,|x|} \ge
    |X_{n:k+1}|\}.
  \end{align*}
  For $s\gg_{\rx, \delta} 1$,
  \begin{align*}
    \sum_{n\le A(\delta s)} 
    \sup_{|x|\ge s}\pr{E_{n,x}\AND  |S_{n:k}(X)|\le (1-\rx) |x|} 
    \ll_{h,k}
    \frac{A(\delta s)}{s^d}
    \int_{1/(2\delta)}^\infty u^{d-1} \exf{-\rx u^{1-\gamma}} \,\dd
    u.
  \end{align*}
  In particular, if $k\ge 1$,
  \begin{align*}
    \sum_{n\le A(\delta s)} \sup_{|x|\ge s}
    \pr{E_{n,x}\AND |X_{n:1}|\le (1-\rx) |x|/k} 
    \ll_{h,k}
    \frac{A(\delta s)}{s^d}
    \int_{1/(2\delta)}^\infty u^{d-1} \exf{-\rx u^{1-\gamma}} \,\dd
    u.
  \end{align*}
\end{lemma}

For the next two lemmas, define
\begin{align*}
  S'_n(X) =\sum_{i=1}^n X_i \cf{|X_i|> a_n}.
\end{align*}
\begin{lemma} \label{l:Poisson}
  Given $0<\delta<\rx<1$, for $s\gg_{\rx,\delta} 1$,
  \begin{align*}
    &
    \sum_{n\le A(\delta s)}\sup_{|x|\ge s}
    \pr{S_n(X)\in x + h I_d\AND
      |S_n(X) - S'_n(X)| \ge \rx|x|}
    \ll_h
    \frac{A(\delta s)}{s^d}
    \int_{1/(2\delta)}^\infty u^{d-1} \exf{-\rx u}\,\dd u.
  \end{align*}
\end{lemma}

\begin{lemma} \label{l:bulk}
  Fix $0<\theta<\nu<1$ and $0< \eta < \rx < \nu - \theta$.
  For $s\gg_{\theta, \nu, \eta, \rx, h} 1$,
  \begin{align*}
    &
    \pr{
      S_n(X)\in s\omega + h I_d\AND
      |X_{n:1}|>\nu s \AND
      |S_n(X) - S'_n(X)|<\eta s
    }
    \ll_h
    n a^{-d}_n\sup_{|t|>\theta s} K(t,a_n, \rx/\theta, h).
  \end{align*}
\end{lemma}

\begin{proof}[Proof of Theorem \ref{t:small-n}]
  Since $K(t,a, \eta, h)$ is increasing in $\eta$, if \eqref
  {e:small-n-local} holds for some $\eta>0$, then it holds for all
  larger $\eta$ with everything else unchanged.  Therefore, to prove
  \eqref {e:small-n-local}, $\eta>0$ can be fixed as small as desired.
  Noting $0<\theta \kappa<1$, let $\eta < 1/(\theta\kappa)-1$.  Fix
  $\nu = \nu(\eta,\theta) \in ((1+\eta) \theta, 1/\kappa)$.  Then
  $\sup_\omega F^{*n}(s\omega + h I_d) \le \sum_{i=1}^5 C_{n,i}(s)$,
  where the supremum is taken over $\omega \in S^{d-1}$ and
  \begin{align*}
    C_{n,1}(s)
    &=
    \sup_\omega
    \pr{S_n(X)\in s\omega + h I_d\AND |X_{n:\kappa+1}|> \zeta_{n,s}},
    \\
    C_{n,2}(s)
    &=
    \sup_\omega
    \pr{S_n(X)\in s\omega + h I_d\AND |X_{n:1}| \le \zeta_{n,s}},
    \\
    C_{n,3}(s)
    &=
    \sup_\omega\pr{S_n(X)\in s\omega + h I_d\AND
      |X_{n:\kappa+1}| \le \zeta_{n,s}< |X_{n:1}|\le \nu s},
    \\
    C_{n,4}(s)
    &=
    \sup_\omega
    \pr{S_n(X)\in s\omega + h I_d\AND |S_n(X) - S'_n(X)|\ge 0.9\theta
      \eta s},
    \\
    C_{n,5}(s)
    &=
    \sup_\omega
    \pr{S_n(X)\in s\omega + h I_d\AND |X_{n:1}| > \nu s\AND
      |S_n(X) - S'_n(X)| < 0.9 \theta \eta s}.
  \end{align*}
  Put $\rx_n(s) = \sum_{i=1}^4 C_{n,i}(s)$.  Apply Lemma
  \ref{l:small-n-many-large} to bound $\sum_{n\le A(\delta s)}
  \sup_{r\ge s} C_{n,1}(r)$ and Lemma \ref{l:small-jump} with $k=0$ to
  bound $\sum_{n\le A(\delta s)} \sup_{r\ge s} C_{n,2}(r)$.  Fixing
  $\rx'>0$ such that $\nu\le (1-\rx')/\kappa$,
  \begin{align*}
    C_{n,3}(s)
    \le
    \sum^\kappa_{k=1} \sup_\omega\pr{S_n(X) \in s\omega + h I_d\AND
      |X_{n:k+1}|\le \zeta_{n,s}<|X_{n:k}|\AND
      |X_{n:1}|\le (1-\rx')s/k      
    }.
  \end{align*}
  Then apply Lemma \ref{l:small-jump} with $k\ge 1$ to bound $\sum_{n
    \le A(\delta s)} \sup_{r\ge s} C_{n,3}(r)$.  Letting $0<\delta <
  0.9\theta\eta$, apply Lemma \ref{l:Poisson} to bound $\sum_{n
    \le A(\delta s)} \sup_{r\ge s} C_{n,4}(r)$.  Together, these
  bounds yield \eqref {e:small-n-r}.  Finally, let $\tilde\eta = 0.9
  \theta \eta$ and $\tilde\rx = \theta\eta$.  Then $0 < \tilde\eta
  < \tilde\rx < \nu-\theta$, so by Lemma \ref{l:bulk}, for
  $s\gg_{\theta, \eta,\rx,h} 1$,
  \begin{align*}
    C_{n,5}(s)
    &= 
    \sup_\omega\pr{
      S_n(X)\in s\omega + h I_d\AND
      |X_{n:1}|>\nu s \AND
      |S_n(X) - S'_n(X)|<\tilde\eta s
    }
    \\
    &\ll_h
    n a^{-d}_n\sup_{|t|>\theta s} K(t,a_n, \tilde\rx/\theta, h),
  \end{align*}
  yielding the first term on the RHS of \eqref{e:small-n-local}.
\end{proof}

\begin{proof}[Proof of Lemma \ref{l:small-n-many-large}]
  Denote $f_n(x) = \pr{S_n(X)\in x + h I_d\AND |X_{n:k}|>
    \zeta_{n,|x|}}$.  Clearly, if $n<k$, then $f_n(s\omega)=0$.  For
  $n\ge k$, $s>0$ and $\omega\in S^{d-1}$,
  \begin{align*}
    f_n(s\omega)
    \le 
    n^k \pr{S_n(X) \in s\omega + h I_d\AND |X_{k:k}|>\zeta_{n,s}}.
  \end{align*}
  For any $z_i\in\ssp$, $\pr{S_n(X)\in s\omega+h I_d\gv X_i=z_i,\AND
    i\le k} = \pr{S_{n-k}(X)\in s\omega-S_k(z) + h I_d}$, which by
  Proposition \ref{p:llt} is $O_h(a^{-d}_{n-k})$.  Then
  \begin{align}
    \sum_{n\le A(\delta s)} \sup_{|x|\ge s} f_n(x)
    \ll_h
    \sum_{n\le A(\delta s)}
    a^{-d}_{n-k} n^k \pr{|X_{k:k}|> \zeta_{n,s}} 
    \ll
    \sum_{n\le A(\delta s)}
    a^{-d}_n n^k q_X(\zeta_{n,s})^k.
    \label{e:small-n-many-large}
  \end{align}
  For $s\gg_\delta 1$, since $q_X(s) \ll 1/A(s)$, 
  \begin{align*}
    \sum_{n\le A(\delta s)} 
    a^{-d}_n n^k q_X(\zeta_{n,s})^k
    &\ll
    \sum_{n\le A(\delta s)} 
    \frac{n^k}{a^d_n A(\zeta_{n,s})^k}
    \\
    &\asymp
    \int_1^{A(\delta s)} 
    \frac{t^k\,\dd t}{(A^{-1}(t))^d A(A^{-1}(t)^{1-\gamma}
      s^\gamma)^k}
    =
    \int_{a_1}^{\delta s} 
    \frac{A(u)^k A'(u)\,\dd u}{u^d A(u^{1-\gamma} s^\gamma)^k},
  \end{align*}
  where the last line is due to change of variable $t=A(u)$.  By
  $A'(u) \asymp A(u)/u$, for $s\gg_\delta 1$,
  \begin{align*}
    \sum_{n\le A(\delta s)} 
    a^{-d}_n n^k q_X(\zeta_{n,s})^k
    &\ll
    \int_{a_1}^{\delta s} 
    \frac{A(u)^{k+1}\,\dd u}
    {u^{d+1} A(u^{1-\gamma} s^\gamma)^k }
    \ll
    A(\delta s) 
    \int_{a_1}^{\delta s} 
    \nth{u^{d+1}}
    \Sbr{\frac{A(u)}{A(u^{1-\gamma} s^\gamma)}}^k\dd u.
  \end{align*}
  
  For $a_1\le u\le \delta s$, since $u<u^{1-\gamma} s^\gamma$ and
  $b\in (0,\alpha)$, by Potter's Theorem (\cite{\Bing}, Th.~1.5.6),
  $A(u)/A(u^{1-\gamma} s^\gamma) \ll [u/(u^{1-\gamma} s^\gamma)]^b =
  (u/s)^{b\gamma}$.  Then by $k\gamma b>d$,
  \begin{align*}
    \int_{a_1}^{\delta s} 
    \nth{u^{d+1}}
    \Sbr{\frac{A(u)}{A(u^{1-\gamma} s^\gamma)}}^k\dd u
    \ll
    \int_0^{\delta s}
    \frac{(u/s)^{k\gamma b}}{u^{d+1}}\,\dd u
    \ll
    \nth{s^d} \frac{\delta^{k\gamma b-d}}{k\gamma b-d}.
  \end{align*}
  By $A(\delta s) \asymp \delta^\alpha A(s)$ for $s\gg_\delta 1$, the
  above two displays together imply
  \begin{align*}
    \sum_{n\le A(\delta s)} a^{-d}_n n^k q_X(\zeta_{n,s})^k
    \ll
    \frac{A(s)}{s^d}
    \frac{\delta^{k\gamma b+\alpha-d}}{k\gamma b-d}.
  \end{align*}
  This combined with \eqref{e:small-n-many-large} finishes the proof.
\end{proof}

\begin{proof}[Proof of Lemma \ref{l:small-jump}]
  Fixing $k\ge 0$ and $\rx,\delta\in (0,1)$, put $f_n(x) =\pr{E_{n,x}
    \AND |S_{n:k}(X)|\le (1-\rx) |x|}$.  If $n<k$, then $E_{n,x} =
  \emptyset$ and so $f_n(x)=0$.  Let $n\ge k$.  Define
  \begin{align*}
    g_n(x)
    =
    \Pr{
      S_n(X)\in x + h I_d\AND |S_k(X)|\le (1-\rx)|x|\AND
      |X_i|>\zeta_{n,|x|} \ge |X_j|\AND i\le k<j
    }.
  \end{align*}
  Then $f_n(x) \le n^k g_n(x)$.  Let $Y_j = X_{k+j}$.  By $S_n(X) =
  S_k(X) + S_{n-k}(Y)$, for $s>0$ and $\omega\in S^{d-1}$, if
  $S_n(X)\in s\omega + h I_d$ and $|S_k(X)|\le (1-\rx)s$,  then
  $|S_{n-k}(Y)|\ge |S_n(X)| - |S_k(X)| \ge \rx s - h\sqrt d$, so
  \begin{align*}
    g_n(s\omega) 
    &\le 
    \Pr{
      \begin{array}{c}
        S_{n-k}(Y) \in s\omega - S_k(X) + h I_d\AND
        |S_{n-k}(Y)|\ge \rx s - h\sqrt d,\\
        \text{and}\ |X_{k:k}|> \zeta_{n,s} \ge |Y_{n-k:1}|
      \end{array}
    }.
  \end{align*}
  By conditioning on $X_i=z_i$, $i\le k$, with $|z_i|> \zeta_{n,s}$,
  \begin{align*}
    g_n(s\omega)
    \le
    \pr{|X_{k:k}| > \zeta_{n,s}} M_{n,s}
    = q_X(\zeta_{n,s})^k M_{n,s},
  \end{align*}
  where
  \begin{align*}
    M_{n,s}
    =
    \sup_{|y|\ge \rx s - h\sqrt d} \pr{S_{n-k}(Y)\in y + h I_d\AND
      |Y_{n-k:1}|\le \zeta_{n,s}}.
  \end{align*}
  By Proposition \ref{p:local-LDP}, there is $C>0$ that only depends
  on $\{F, A\}$, such that
  \begin{align*}
    M_{n,s}
    \ll_h
    (\zeta^{-d}_{n, s} + a^{-d}_{n-k})
    \exf{-\rx s/\zeta_{n,s} + C n/A(\zeta_{n,s})}
  \end{align*}
  for $s\gg_\rx 1$ and $n\ge k$.  For $n\le A(\delta s)$, as
  $\zeta_{n, s} = a_n^{1-\gamma} s^\gamma \ge a_n$, $M_{n,s} \ll_{h,k}
  a^{-d}_n \exf{-\rx (s/a_n)^{1-\gamma}}$.  As a result, $g_n(s\omega)
  \ll_{h,k} a^{-d}_n q_X(\zeta_{n,s})^k \exf{-\rx(s/a_n)^{1 -
      \gamma}}$, and hence
  \begin{align}
    f_n(s\omega)
    \le
    n^k g_n(s\omega)
    \ll_{h,k}
    a_n^{-d} n^k q_X(\zeta_{n,s})^k \exf{-\rx(s/a_n)^{1-\gamma}}
    \ll_{h,k}
    a_n^{-d} \exf{-\rx(s/a_n)^{1-\gamma}},
    \label{e:small-jump-LDP2}
  \end{align}
  where the last bound is due to $n q_X(\zeta_{n,s}) \ll
  A(a_n)/A(\zeta_{n,s}) \le 1$.  Take sum over $n\le A(\delta s)$.
  Since for $n\ge 1$ and $t\in [n,n+1]$, $a_n\le A^{-1}(t)\ll a_n$,
  then for $s\gg_\delta 1$,
  \begin{align*}
    \sum_{n\le A(\delta s)} \sup_{|x|\ge s} f_n(x)
    \ll_{h,k}
    \int_1^{A(2\delta s)} \nth{A^{-1}(t)^d}
    \exf{-\rx(s/A^{-1}(t))^{1-\gamma}}\,\dd t.
  \end{align*}
  By change of variable $t=A(s/u)$, or $u=s/A^{-1}(t)$, and use $A'(x)
  \asymp A(x)/x$ for $x>0$, the last integral is no greater than
  \begin{align*}
    \int_{1/(2\delta)}^\infty \nth{(s/u)^d} \exf{-\rx u^{1-\gamma}}
    A'(s/u) s u^{-2}\,\dd u
    &\ll
    s^{-d} \int_{1/(2\delta)}^\infty u^{d-1} \exf{-\rx u^{1-\gamma}}
    A(s/u)\,\dd u
    \\
    &\le
    s^{-d} A(2\delta s)
    \int_{1/(2\delta)}^\infty u^{d-1} \exf{-\rx u^{1-\gamma}}
    \,\dd u.
  \end{align*}
  Combining the above two displays then finishes the proof.
\end{proof}

To prove Lemmas \ref{l:Poisson} and \ref{l:bulk}, define
\begin{align*} 
  \tau_n =\sum_{i=1}^n \cf{|X_i|> a_n}.
\end{align*}
Then $\pr{\tau_n = m} = \binom{n}{m} q_X(a_n)^m [1-q_X(a_n)]^{n-m}$
and by $q_X(a_n)\ll 1/A(a_n)=1/n$,
\begin{align} \label{e:tau-bound}
  \pr{\tau_n = m}
  \le
  \frac{n! O(1/n)^m}{m!(n-m)!}
  \ll
  \frac{O(1)^m}{m!}.
\end{align}
Conditioning on $\tau_n = m$,
\begin{align*}
  (S_n(X) - S'_n(X), S'_n(X)) \sim (S_{n-m}(b\Sp n), S_m(u\Sp n)),
\end{align*}
where $\inum {b\Sp n}$, $\inum {u\Sp n}$ are independent, with
\begin{align*}
  \pr{b\Sp n_i\in \dd x}
  &=
  \begin{cases}
    \pr{X\in \dd x\gv |X|\le a_n} & \text{if } q_X(a_n)<1
    \\
    \delta_0(\dd x) & \text{else},
  \end{cases}
  \\
  \intertext{and}
  \pr{u\Sp n_i\in \dd x}
  &=
  \begin{cases}
    \pr{X\in \dd x\gv |X|>a_n} & \text{if } q_X(a_n)>0
    \\
    \delta_0(\dd x) & \text{else},
  \end{cases}
\end{align*}
where $\delta_0$ is the unit measure concentrated at 0.

\begin{proof}[Proof of Lemma \ref{l:Poisson}]
  Denote $f_n(x) = \pr{S_n(X)\in x + h I_d\AND |S_n(X) - S'_n(X)| \ge
    \rx |x|}$.  For $n\le A(\delta s)$ and $\omega\in S^{d-1}$,
  \begin{align}
    f_n(s\omega)
    =
    \sum_{m=0}^n
    \pr{S_{n-m}(b\Sp n) + S_m(u\Sp n)\in s\omega + h I_d\AND
      |S_{n-m}(b\Sp n)|\ge \rx s} \pr{\tau_n=m}.
    \label{e:Poisson}
  \end{align}
  For each $m\le n$,
  \begin{align}
    &
    \pr{S_{n-m}(b\Sp n) + S_m(u\Sp n)\in s\omega + h I_d\AND
      |S_{n-m}(b\Sp n)|\ge \rx s} 
    \nonumber \\
    &=
    \int \pr{S_{n-m}(b\Sp n) \in s\omega - z + h I_d\AND
      |S_{n-m}(b\Sp n)|\ge \rx s} \pr{S_m(u\Sp n)\in \dd z}
    \nonumber \\
    &\le
    \sup_{|x|\ge \rx s - h\sqrt d} \pr{S_{n-m}(b\Sp n)\in x + h I_d}.
    \label{e:Poisson2}
  \end{align}
  For $n$ with $q_X(a_n)<1$,
  \begin{align*}
    \pr{S_{n-m}(b\Sp n)\in x + h I_d}
    &=
    \frac{
      \pr{S_{n-m}(X)\in x + h I_d\AND |X_{n-m:1}|\le a_n}
    }{
      \pr{|X_{n-m:1}|\le a_n}
    }.
  \end{align*}
  Since $q_X(a_n) \ll 1/n$, for all $n\ge 1$ with $q_X(a_n) < 1$ and
  $1\le m\le n$, $\pr{|X_{n-m:1}| \le a_n} \ge [1-q_X(a_n)]^n \gg 1$.  
  Then
  \begin{align*}
    \pr{S_{n-m}(b\Sp n)\in x + h I_d}
    \ll
    \pr{S_{n-m}(X)\in x + h I_d\AND |X_{n-m:1}|\le a_n}.
  \end{align*}
  Let $s_0>0$ be as in Proposition \ref{p:local-LDP}.  Let $n_0$ be
  the largest $n$ with $a_n<s_0$.  For $n\le n_0$, if $|x|> n a_n + h
  \sqrt d$, then the RHS is 0.  Otherwise, as $|x| \ll 1$, the RHS is
  $O(a^{-d}_{n-m} e^{-|x|/a_n})$.  On the other hand, for $n>n_0$,
  since $a_n\ge s_0$, by applying Proposition \ref{p:local-LDP} to the
  RHS
  \begin{align} \label{e:S-n-m-a}
    \pr{S_{n-m}(b\Sp n)\in x + h I_d}
    \ll_h
    (a^{-d}_{n-m} + a^{-d}_n) \exf{-|x|/a_n + C (n-m)/A(a_n)}
    \ll_h
    a^{-d}_{n-m} \exf{-|x|/a_n}.
  \end{align}
  From the discussion for $n\le n_0$, it is seen the bound holds for
  all $n$ with $q_X(a_n)<1$.  For $n$ with $q_X(a_n)=1$, as $b\Sp
  n_i\equiv 0$, $\pr{S_{n-m}(b\Sp n) \in x + h I_d} = \cf{0\in x + h
    I_d} \le \cf{|x| \le h\sqrt d}$.  Since there are only a
  finite number of $n$ with $q_X(a_n)=1$, the bound in \eqref
  {e:S-n-m-a} still holds.  Combining the bound with \eqref
  {e:Poisson}--\eqref {e:Poisson2}, for $|x|\ge\rx s - h\sqrt d$,
  \begin{align*}
    f_n(s\omega)
    \ll_h
    \exf{-\rx s/a_n} \sum_{m=0}^n a^{-d}_{n-m} \pr{\tau_n=m}
  \end{align*}
  Since by \eqref{e:tau-bound},
  \begin{align*}
    \sum_{m=0}^n a^{-d}_{n-m} \pr{\tau_n=m}
    &\ll
    a^{-d}_n \sum_{m\le n/2} \pr{\tau_n=m}
    + a^{-d}_0 \sum_{n/2<m\le n} \frac{O(1)^m}{m!}
    \ll
    a^{-d}_n + \frac{O(1)^n}{\Flr{n/2}!} \ll a^{-d}_n,
  \end{align*}
  then,
  \begin{align*}
    \sum_{n\le A(\delta s)} \sup_{|x|\ge s}
    f_n(x)
    \ll_h
    \sum_{n\le A(\delta s)} a^{-d}_n \exf{-\rx s/a_n}
  \end{align*}
  The rest of the proof is similar to the argument that starts with
  \eqref{e:small-jump-LDP2} for Lemma \ref{l:small-jump}.
\end{proof}

\begin{proof}[Proof of Lemma \ref{l:bulk}]
  Put $f_n(s\omega) = \pr{S_n(X)\in s\omega + h I_d\AND |X_{n:1}| >
    \nu s \AND|S_n(X) - S'_n(X)|<\eta s}$.  For $s\gg_{\eta, h} 1$, if
  $S_n(X)\in s\omega + h I_d$ and $|S_n(X) - S'_n(X)|<\eta s$, then
  $S'_n(X)\ne 0$, yielding $\tau_n\ge 1$ and $|X_{n:1}| > a_n$.
  Therefore, if $q_X(a_n)=0$, then $f_n(s\omega)=0$ and the bound in
  Lemma \ref{l:bulk} trivially holds.  In the rest of the proof, let
  $q_X(a_n)>0$.  Then
  \begin{align}
    f_n(s\omega)
    \le
    \sum_{m=1}^n P_m(s\omega) \pr{\tau_n = m},
    \label{e:truncate}
  \end{align}
  where, with $n$ being fixed, for each $m=1,\ldots, n$,
  \begin{align*}
    P_m(s\omega)
    &=
    \pr{S_{n-m}(b\Sp n) + S_m(u\Sp n)\in s\omega + h I_d\AND
      |u\Sp n_{m:1}|>\nu s\AND
      |S_{n-m}(b\Sp n)| < \eta s
    }
    \\
    &=
    \int_{|y|<\eta s}
    \pr{S_m(u\Sp n)\in s\omega - y + h I_d\AND
      |u\Sp n_{m:1}|>\nu s} \pr{S_{n-m}(b\Sp n)\in\dd y}
    \\
    &\le
    m
    \int_{|y|<\eta s}
    \pr{S_m(u\Sp n)\in s\omega - y + h I_d\AND
      |u\Sp n_1|>\nu s} \pr{S_{n-m}(b\Sp n)\in\dd y}.
  \end{align*}
  Denote $T = s\omega - (u\Sp n_2 + \cdots + u\Sp n_m)$.  By
  independence of $T$ and $u\Sp n_1$, with the latter following the
  distribution of $X$ conditioned on $|X|>a_n$,
  \begin{align*}
    \pr{S_m(u\Sp n)\in s\omega - y + h I_d\AND |u\Sp n_1|>\nu s}
    &=
    \pr{X \in T - y + h I_d\AND |X|>\nu s\gv |X|> a_n}
    \\
    &\le
    \frac{\pr{X \in T - y + h I_d\AND |X|>\nu s}}{q_X(a_n)}.
  \end{align*}
  For $y$ with $|y|<\eta s$, if $X\in T-y + h I_d$ and $|X|>\nu
  s$, then $|T| \ge |X|-|y|-h\sqrt d > (\nu - \eta)s- h\sqrt d$.
  For $s\gg_{\eta, \nu, \theta, h} 1$, $(\nu-\eta) s - h\sqrt d> \theta
  s$ and hence
  \begin{align*}
    P_m(s\omega)
    &\le
    \frac{m}{q_X(a_n)} \int_{|y|<\eta s}
    \pr{X\in T - y + h I_d\AND |X|>\nu s}
    \pr{S_{n-m}(b\Sp n)\in \dd y}
    \\
    &\le
    \frac{m}{q_X(a_n)} \int_{|y|<\eta s}
    \pr{X\in T - y + h I_d\AND |T|>\theta s}
    \pr{S_{n-m}(b\Sp n)\in \dd y}
  \end{align*}
  Let
  \begin{align*}
    G_{n,m}(t,s) = 
    \int_{|y|<\eta s}
    F(t - y + h I_d)\,\pr{S_{n-m}(b\Sp n)\in \dd y}.
  \end{align*}
  Then by Fubini's theorem, the last inequality yields
  \begin{align}
    P_m(s\omega)
    &\le
    \frac{m}{q_X(a_n)} \int_{|t|>\theta s} 
    G_{n,m}(t,s) \pr{T\in \dd t}
    \le
    \frac{m}{q_X(a_n)}\sup_{|t|> \theta s} G_{n,m}(t,s).
    \label{e:conv-max}
  \end{align}
  
  Given $v\in h I_d$, $\eta s B_d$ is covered by disjoint cubes $z
  + h I_d$ with $z\in (v+h\Ints^d)\cap (\eta s+ h \sqrt d) B_d$.  If
  $y\in z + h I_d$, then $- y + h I_d \subset - z + h J_d$, where $J_d
  = (-1, 1)^d$.  As a result, for any $t\in\ssp$, 
  \begin{align*}
    G_{n,m}(t,s)
    &\le
    \sum_{z\in (v+h\Ints^d)\cap(\eta s + h\sqrt d) B_d}
    \int_{z + h I_d} F(t-y + h I_d) \pr{S_{n-m}(b\Sp n)\in \dd y}
    \nonumber \\
    &\le
    \sum_{z\in (v+h \Ints^d) \cap (\eta s + h\sqrt d) B_d}
    F(t - z + h J_d) \pr{S_{n-m}(b\Sp n)\in z + h I_d}.
  \end{align*}
  Then applying \eqref{e:S-n-m-a} to $\pr{S_{n-m}(b\Sp n)\in z + h
    I_d}$,
  \begin{align*}
    G_{n,m}(t,s)\ll_h
    a^{-d}_{n-m}
    \sum_{z\in (v+h\Ints^d)\cap (\eta s + h\sqrt d)B_d}
    F(t - z + h J_d) \exf{-|z|/a_n}.
  \end{align*}
  Let $u=z-v$.  Then $z\in (v + h \Ints^d) \cap(\eta s + h\sqrt d)B_d$
  implies $u\in (h \Ints^d)\cap (\eta s + 2 h\sqrt d) B_d$, yielding
  \begin{align*}
    G_{n,m}(t,s)\ll_h
    a^{-d}_{n-m}
    \sum_{u\in (h\Ints^d)\cap (\eta s + 2 h\sqrt d)B_d}
    F(t - u-v + h J_d) \exf{-|u+v|/a_n}.
  \end{align*}
  Take average over $v\in h I_d$.  By $z=u+v$ and Fubini's theorem,
  \begin{align*}
    G_{n,m}(t,s)
    &\ll_h
    a^{-d}_{n-m}
    \sum_{u\in (h\Ints^d)\cap (\eta s + 2h\sqrt d)B_d}
    \int_{z\in u+h I_d} F(t - z + h J_d)
    \exf{-|z|/a_n}\dd z
    \\
    &\le
    a^{-d}_{n-m}
    \int_{(\eta s + 3h\sqrt d)B_d}
    F(t - z + h J_d) \exf{-|z|/a_n}\, \dd z
    \\
    &\ll_h
    a^{-d}_{n-m}
    \int_{(\eta s + 4 h\sqrt d)B_d}
    F(t - z + h I_d) \exf{-|z|/a_n}\, \dd z.
  \end{align*}
  Now for $s\gg_{\eta,\rx, h} 1$, if $|t|>\theta s$ and $z\in (\eta s
  + 4 h\sqrt d) B_d$, then $|z|<(\rx/\theta) |t|$ and so the last
  integral is no greater than $K(t, a_n, \rx/\theta, h)$.  Combining
  the bound with \eqref{e:conv-max} and then with \eqref{e:truncate},
  \begin{align*}
    f_n(s\omega)
    \le
    \nth{q_X(a_n) }
    \sup_{|t|> \theta s} K(t, a_n, \rx/\theta, h)
    \sum_{m=1}^n m a^{-d}_{n-m} \pr{\tau_n = m}.
  \end{align*}
  Similar to the argument at the end of the proof of Lemma
  \ref{l:Poisson},
  \begin{align*}
    \sum_{m=1}^n m a^{-d}_{n-m} \pr{\tau_n = m}
    &\ll
    a^{-d}_n\sum_{1\le m\le n/2+1} m \pr{\tau_n = m}
    +
    \sum_{n/2+1<m\le n} m \pr{\tau_n = m}
    \\
    &\ll 
    a^{-d}_n \mean(\tau_n) + \sum_{n/2+1<m\le n}
    \frac{n! q_X(a_n)^m}{(m-1)!(n-m)!}.
    \\
    &\ll
    a_n^{-d} n q_X(a_n) + n q_X(a_n)
    \sum_{n/2<m\le n} \frac{O(1)^{m-1}}{(m-1)!}
    \\
    &\ll
    n q_X(a_n)\Grp{a^{-d}_n  + \frac{O(1)^n}{\Flr{n/2}!}}
    \ll n q_X(a_n) a^{-d}_n.
  \end{align*}
  Combining the above two displays, the proof is complete.
\end{proof}

\section{Proofs of other results on the SRT} \label{s:proof-props}

\begin{proof}[Proof of Theorem \ref{t:w-SRT}]
  Let $X\sim F$.  For $t\in\ssp$ with $|t|\gg_h 1$ and $a>0$,
  \begin{align*} 
    K(t, a, 1/3,h)
    &\le
    \int_{|z|\le |t|/3} \dd z \int \cf{x\in t-z+ h I_d} F(\dd x)
    \\
    &=
    \int_{|z|\le |t|/3}\dd z
    \int_{|x|>|t|/3} \cf{x\in t-z+ h I_d} F(\dd x)
    \\
    &\le
    \int_{|x|>|t|/3}  F(\dd x)
    \int \cf{x-t+z\in h I_d}\,\dd z
    = h^d q_X(|t|/3). 
  \end{align*}
  Then, for any $\theta>0$ and $\eta\le 1/3$, by
  $K(t, a, \eta,h)\le K(t,a,1/3,h)$,
  \begin{align} \label{e:K}
    \sum_{n\le A(\delta s)}
    n a^{-d}_n\sup_{|t|>\theta s} K(t, a_n, \eta, h)
    \ll q_X(s/3) \tilde A_d(\delta s),
  \end{align}
  where $\tilde A_d$ is defined in \eqref{e:L}.  Since $\alpha\in (0,
  2)$, the RHS is $O(\tilde A_d(\delta s)/A(s)) \asymp \delta^{2\alpha
    - d} A(s)/s^d$, and hence the proof follows from Theorem \ref
  {t:SRT}.
\end{proof}

\begin{proof}[Proof of Proposition \ref{p:normal}]
  The inequality in \eqref{e:K} still holds but now $q_X(s) =
  o(1/A(s))$ as $s\toi$ (cf.\ \eqref{e:prelim-q-V}).  By Theorem \ref
  {t:SRT}, it suffices to verify $q_X(s) \tilde A_d(s)
  =o(A(s)/s^d)$ in each case.  The value of $\delta$ is irrelevant.
  If $d=3$, then by $\alpha=2>d/2$, $q_X(s) \tilde A_3(s) \asymp
  q_X(s) A(s)^2 / s^3 = o(A(s)/s^3)$.  If $d=4$, then the proof
  directly follows from $\tilde A_4(s) \ll \int_1^s u^{-5} A(u)^2\,\dd
  u$ (cf.\ \eqref{e:L}).  Finally, if $d\ge 5$, then $\tilde
  A_d(\delta s) \le \tilde A_d(\infty) < \infty$ and $A(s)\asymp s^2$
  for $s\gg 1$.  So $q_X(s) = o(s^{2-d})$ implies $q_X(s) \tilde
  A_d(s) = o(A(s)/s^d)$. 
\end{proof}

\begin{proof}[Proof of Theorem \ref{t:bounded-ratio}]
  Since the integral in \eqref{e:bounded-ratio} is increasing in
  $\eta$, assume without loss of generality that $0<\eta<1$ in
  \eqref{e:bounded-ratio}.  Fix $\theta>0$.  For $s\gg 1$, $n\ge 1$,
  and $t$, $z\in\ssp$ with $|t|\ge \theta s$ and $|z|\le \eta|t|$,
  $|t-z|^d A(|t-z|) \asymp |t|^d A(|t|)$.  Then
  \begin{align*}
    K(t,a,\eta,h)
    &\ll
    \nth{|t|^d A(|t|)}\int_{|z|<\eta |t|} \phi(t - z)
    \exf{-|z|/a}\,\dd z
    \\
    &\ll
    \nth{|t|^d A(|t|)}\Sbr{
      \int_{|z|<\eta |t|} [\phi(t - z)-T]_+\,\dd z
      +
      T\int\exf{-|z|/a}\,\dd z
    }
    \\
    &\le
    o(1) \frac{A(|t|)}{|t|^d} + \frac{O(a^d)}{|t|^d A(|t|)},
  \end{align*}
  which combined with \eqref{e:L} yields \eqref{e:bounded-ratio-K}.
\end{proof}

\begin{proof}[Proof of Proposition \ref{p:radial}]
  Let $Y\in\Reals$ be independent of $\xi$ with $\chf_Y(\theta) =
  \exf{-C f_\alpha(\theta)}$.  As noted before Proposition \ref
  {p:radial}, $Y$ is strictly stable.  It actually has \levy density
  $c \cf{x>0} x^{-1-\alpha}$ for some constant $c>0$.  Let $X=Y\xi$.
  Then $q_X(s) = q_Y(s) \asymp s^{-\alpha}$ for $s\gg 1$. 
  For $\Gamma\subset S^{d-1}$, $\pr{|X| > s\AND X/|X|\in \Gamma} =
  \pr{Y>s} \pr{\xi\in \Gamma} + \pr{Y<-s} \pr{\xi\in -\Gamma}$, so by
  $\pr{Y<-s}\le \mean\exf{-s-Y} \ll \exf{-s}$ (\cite{\Sato},
  Th.~25.17)
  \begin{align*}
    \pr{|X|>s\AND X/|X|\in \Gamma}/q_X(s) \to \pr{\xi\in \Gamma},
    \quad 
    s\toi.
  \end{align*}
  Then $X$ is in the domain of attraction of a stable law with the
  same \levy measure as $G$, with $a_n = n^{1/\alpha}$ being
  norming constant (\cite{\Rva}, Th.~4.2).  By $Y\in \DA_0(\alpha)$
  and \eqref{e:centering}, $(n/a_n) c_X(a_n) = (n/a_n) c_Y(a_n)
  \mean\xi$ converges, so by \eqref{e:centering} again, $S_n(X)/a_n$
  weakly converges to a strictly stable law.  Since $G$ is strictly
  stable, if $\alpha\ne 1$, then the limiting law is $G$.  However, if
  $\alpha=1$, then the limiting law is $G(x-x_0)$, where $x_0$ need
  not be 0.  Let $g$ be the density of $Y$ and $\lambda$ be the
  density of $\xi$ with respect to the spherical measure $\sigma$ on
  $S^{d-1}$.   Then for
  $E\subset\ssp$,
  \begin{align*}
    \pr{X\in E}
    &=
    \int\cf{y u\in E} g(y)\lambda(u) \dd y\,\sigma(\dd u)
    \\
    &=
    \int_{r>0}\cf{ru\in E} [g(r)\lambda(u) + g(-r)\lambda(-u)]\dd r\,
    \sigma(\dd u).
  \end{align*}
  For $x=ru\ne 0$ with $r=|x|$, letting $h(x) = c[g(r) \lambda(u) +
  g(-r) \lambda(-u)]/r^{d-1}$ with $c>0$ a suitable constant, the last
  integral is equal to $\int \cf{x\in E} h(x)\,\dd x$, showing that
  $X$ has density $h$.  Since $\sup_u [r^d h(ru)/q_X(r)]\ll
  (\sup\lambda) r^{1+\alpha}[g(r) + g(-r)]\ll 1$, it is seen that the
  function $\phi(x)$ in Theorem \ref{t:bounded-ratio} is bounded and
  hence \eqref{e:uniform-SRT2} holds for the limiting law of
  $S_n(X)/a_n$.  For $\alpha\ne 1$, this completes the proof.  If
  $\alpha=1$, one can only conclude that \eqref{e:uniform-SRT2} holds
  for $G(x-x_0)$.  However, consider $X + x_0$, whose corresponding
  limiting law is $G$.  Since $|x|^d \pr{X + x_0\in x + h I_d} A(|x|)
  \ll |x_0|^d A(|x_0|) + |x-x_0|^d \pr{X\in x - x_0 + h I_d} A(|x -
  x_0|)$ is bounded, a repeat of argument shows that \eqref
  {e:uniform-SRT2} holds for $G$.
\end{proof}

\begin{proof}[Proof of Proposition \ref{p:bounded-ratio-ind}]
  It suffices to show that $X\sim F$ satisfies \eqref{e:small-n}.
  Part of the argument is similar to that in \cite{chi:14:tr}, so only 
  parts that are different will be shown in detail.  First, the
  support of $\nu$ is unbounded, otherwise $\mean\exf{|t X|} < \infty$
  for all $t$ and $F\in \DA_0(2)$ (\cite{\Sato}, Th.~25.17).  Let
  $\nu_1(\cdot) = \nu(\cdot\setminus B_d)$ and $\lambda = \nu -
  \nu_1$.  Let $\mu = \nu_1(\ssp)$.  Then $S_n(X) \sim S_{S_n(N)}(Z) +
  S_n(W) + n v$, where $N_i$, $Z_j$ and $W_k$ are independent with
  $N_i$ \iid $\sim \dpois(\mu)$, $Z_j$ \iid $\sim \nu_1/\mu$,
  and $W_k$ \iid., ID with \levy measure $\lambda$ and mean 0, and
  $v\in\ssp$ is a constant.  Fix $M>1\vee (4\mu)^{1/\alpha}$ and
  $\rx>0$.  Let $Y = Z + v/\mu$ and $G$ the distribution of $Y$.
  Let $V = W + v - N v/\mu$.  Then
  \begin{align*}
    \Pseq n{s\omega + h I_d}
    &=
    \pr{S_{S_n(N)}(Y) + S_n(V)\in s\omega + h I_d}
    \\
    &\le
    \sup_{|t|<\rx s} \pr{S_{S_n(N)}(Y) \in s\omega - t + h I_d}
    + \pr{|S_n(V)|\ge \rx s}
    \\
    &\le
    \sum_{k\le A(M\delta s)} 
    \pr{S_n(N)=k} \sup_{|y|>(1-\rx) s}
    G^{*k}(y + h I_d) + R_n(s) + R'_n(s),
  \end{align*}
  where $R_n(s) = \pr{S_n(N)>A(M\delta s)}$, $R'_n(s) =
  \pr{|S_n(V)|>\rx s}$.  It can be shown that
  \begin{align*}
    \sum_{n\le A(\delta s)} \Pseq n{s\omega + h I_d}
    \ll
    \sum_{k\le A(M\delta s)} \sup_{|y|>(1-\rx) s}
    G^{*k}(y + h I_d) +
    \sum_{n\le A(\delta s)} [R_n(s) + R'_n(s)].
  \end{align*}
  As in \cite{chi:14:tr}, $\sum_{n\le A(\delta s)} R_n(s) =
  o(A(s)/s^d)$ as $s\toi$.  Let $V=(\eno \xi d)$.  Then $R'_n(s) \le
  \sum_{j=1}^d R'_{n j}(s)$, where $R'_{n j}(s) = \pr{|S_n(\xi_j)|>\rx
    s/d}$.  Each $\xi_j\in\Reals$ has mean zero and $\mean \exf{t
    \xi_j}<\infty$ for all $t$.  Fix $b\in (0, \alpha\wedge 1)$.  For
  $s\gg_b 1/\delta$ and $n\le A(\delta s)$, if $1\le n\le s^b$, then
  \begin{align*}
    R'_{n j}(s)
    &=
    \pr{S_n(\xi_j)>\rx s/d} + \pr{-S_n(\xi_j)>\rx
      s/d}
    \\
    &\le
    [(\mean \exf{\xi_j})^n + (\mean \exf{-\xi_j})^n] \exf{-\rx s/d}
    =
    O(1)^n \exf{-\rx s/d}
    \ll \exf{-\rx s/2d}.
  \end{align*}
  If $s^b<n<A(\delta s)$, then, letting $\sigma_j^2 = \mean[\xi^2_j]$
  and $\sigma = \max \sigma_j$,
  \begin{align*}
    \rx s/(d\sigma\sqrt n)  
    \ge (\rx/\delta) A^{-1}(n)/(d\sigma\sqrt n) \ge \eta n^c
  \end{align*}
  for some constants $\eta>0$ and $0<c<1/6$.  By Cram\'er's large
  deviation (\cite{\Petr}, Th.~5.23), $R'_{n j}(s) \le
  \pr{|S_n(\xi_j)| / (\sigma_j \sqrt n) > \rx s/(d\sigma\sqrt n)} \ll
  1-\Phi(\eta n^c)\le 1-\Phi(\eta s^{bc})$, where $\Phi$ is the
  distribution function of $N(0,1)$.  It is then easy to get
  $\sum_{n\le A(\delta s)} R'_n(s) = o(A(s)/s^d)$.

  Since $N_n/n\conv D\mu$ and $S_n(V)/a_n\conv D 0$, by $S_n(X)/a_n
  \sim S_{N_n}(Y)/a_n + S_n(V)/a_n$, it can be seen that
  $\mu^{1/\alpha} Y$ is in the domain of attraction without centering
  of the same stable law as $X$.  By the assumption on $\phi_\nu(x)$,
  and Theorems \ref{t:bounded-ratio} and \ref{t:small-n},
  \begin{align*}
    \sum_{k\le A(M\delta s)}
    \sup_{|y|>(1-\rx) s} G^{*k}(y + h I_d)
    \ll_h \delta A_\nu(s)/s^d.      
  \end{align*}
  By following almost line by line the argument in \cite
  {feller:71:jws}, p.~572--573, $q_X(s) \sim 1/A_\nu(s)$.  Then
  $A(s) \sim A_\nu(s)$ and the proof is complete.
\end{proof}

\paragraph{Acknowledgment}  The author would like to thank two
referees and the AE for their careful reading of the paper and useful
suggestions.

\comment{
\begin{small}
  \bibliography{LimitTheorems,Levy,Books}
\end{small}
}

\begin{small}

\end{small}

\setcounter{section}{0}
\renewcommand{\thesection}{\Alph{section}}

\section*{Appendix}
\section{Proofs for the lattice-nonlattice composition}
\label{s:lattice} 
For a set $E$ in a Euclidean space, denote by $\lspan(E)$ the linear
subspace spanned by elements of $E$.  If $M$ is a matrix, denote by
$\cspace(M)$ the linear subspace spanned by the column vectors of $M$.
If the rank of $M$ is equal to its number of columns, then $M$ is said
to be of full column rank.

\begin{proof}[Proof of Proposition \ref{p:lattice-decomp}]
  Let $\Gamma = \Gamma_X$, where
  \begin{align} \label{e:Gamma-lattice}
    \begin{split}
      \Gamma_X
      &=
      \{v\in\ssp: \text{there is $a\in \Reals$ such that}\ 
      \ip v X\in a + \Ints \text{ a.s.}\}
      \\
      &=
      \{v\in\ssp: |\chf_X(2\pi v)|=1\}.
    \end{split}
  \end{align}
  As in the proof of \cite{\Spit}, T6.1, $\Gamma$ plays an important
  role.  The first step is to show that it is a lattice.  The first
  line in \eqref{e:Gamma-lattice} implies that $\Gamma$ is an additive
  subgroup of $\ssp$, so it suffices to show that 0 is not a cluster
  point of $\Gamma$.  Let $u_n\in\Gamma$ such that $u_n\to 0$.  Let
  $V_n = \lspan(u_i\AND i\ge n)$.  Since $\ssp\supset V_1\supset V_2
  \supset \ldots$, there is $k$, such that $V_k = V_{k+1} = \cdots$.
  Let $X_*$ be \iid $\sim X$ and $\xi = X - X_*$.  Then almost surely,
  $\ip{u_n}\xi\in \Ints$ for all $n$.  Since $\ip{u_n}\xi\to 0$, this
  implies $\ip{u_n} \xi=0$ for $n\ge k$ large enough.  But then
  $\xi\in V^\perp_n = V^\perp_k$.  By assumption, $\xi$ is not
  concentrated in any linear subspace of dimension $d-1$.  Then $V_k =
  \{0\}$, giving $u_k = u_{k+1} = \cdots = 0$, so 0 is not a cluster
  point of $\Gamma$.

  Let $V=\lspan(\Gamma)$ and $r=\dim(V)$.  Suppose $r\ge 1$.  By a
  fundamental theorem on lattices (cf.~\cite{tao:06:cup}, Lemma 3.4),
  $\Gamma = M\Ints^r$ for some $M\in\Reals^{d\times r}$ of rank $r$.
  Let $\eno v r$ be the column vectors of $M$ and $a=(\eno a r)$ such
  that $\ip{v_i} X\in a_i + \Ints$.  Then $X\in \Lambda = \{x\in\ssp:
  M'x \in a + \Ints^r\}$.  By $\pi_V(x) = H M' x$, where $H =
  M(M'M)^{-1}$, $\pi_V(X) = H M' X \in H(a +\Ints^r)$, so $\pi_V(X)$
  is lattice.  If $v \in \ssp \setminus V$, then $v\not\in\Gamma$,
  so $|\chf_X(2\pi v)|<1$.  Thus $V$ has the property stated in (1).
  To continue, assume the following result is true for now.
  \begin{lemma} \label{l:normalize}
    There is $K\in \Ints^{r\times r}$ with $\det K=\pm 1$ such that
    $K a = (0, \ldots, 0, z_\nu, z_{\nu+1}, \ldots, z_r)$, where
    $z_\nu\in [0,\infty)\cap \Rats$ and $z_{\nu+1}, \ldots, z_r\in
    (0,\infty)\setminus \Rats$ are rationally independent.
  \end{lemma}
  Let $Q\in \Reals^{d\times (d-r)}$ be of full column rank such that
  $Q' M = O$.   Define
  \begin{align} \label{e:lattice}
    T = \begin{pmatrix} K M' \\ Q'\end{pmatrix},\quad
    Y = K M'X, \quad
    Z = Q' X, \quad
    \beta_i = z_i - \Flr{z_i}, \ i=\nu, \ldots, r.
  \end{align}
  Then $TX = (Y, Z)$, $\beta_\nu\in \Rats\cap [0,1)$, and
  $\beta_{\nu+1}$, \ldots,  $\beta_r\in (0,1) \setminus \Rats$ are
  rationally independent.  Put $\beta=(0,\ldots, 0, \beta_\nu,
  \beta_{\nu+1}, \ldots, \beta_r)$.  Since $Y\in K (a + \Ints^r) =
  \beta + \Ints^r$, $T$ has the property stated in (3).

  To show $T$ has the property stated in (2), if $u = (k, 0)\in
  \Ints^r\times \{0\}$, then by $\ip u {T X} = \ip k Y\in 
  \ip k\beta + \Ints$, $|\chf_{T X}(2\pi u)|=1$.  Conversely, if
  $|\chf_{T X}(2\pi u)|=1$, then $|\chf_X(2\pi T' u)|=1$, so $T' u
  = M k\in\Gamma$ for some $k\in \Ints^r$.  Write $u = (w, v)$ with
  $w\in \Reals^r$.  By \eqref{e:lattice}, $H'T' u= H'(M K' w + Q v) =
  K' w$.  As the LHS is also $H'M k=k$, $K' w= k$, giving $w
  = (K')^{-1} k\in\Ints^r$.  On the other hand, $(\Id_d -M H')T' u =
  (\Id_d - M H') (M K' w + Q v) = Qv$ and the LHS is also $(\Id_d -
  MH')M k = 0$.  Thus $Q v=0$.  Since $Q$ is of full column rank,
  $v=0$ and hence $u = (w, 0)\in\Ints^r\times \{0\}$.

  So far it has been assumed that $r=\dim(V)>0$.  If $r=0$, then
  $\Gamma = V = \{0\}$.  Consequently, $|\chf_X(2\pi v)|<1$ for $v\ne
  0$ and $T = \Id_d$ has the property stated in (2)--(3).
  
  To show that $V$ is unique, let $W$ be a linear subspace such
  that $\pi_W(X)$ is lattice and $|\chf_X(2\pi v)|<1$ for $v\not\in
  W$.  By definition, $\Gamma\subset W$, so $V = \lspan(\Gamma)
  \subset W$.  If $V\ne W$, then $W\cap V^\perp\ne\emptyset$ and
  $\pi_{W\cap V^\perp}(X) = \pi_{W\cap V^\perp} (\pi_W(X))$ is
  lattice.  It follows that there is $0\ne u\in W\cap V^\perp$, such
  that $\ip u X = \ip u {\pi_{W\cap V^\perp}(X)} \in c + \Ints$ for
  some $c$.  But then $u\in\Gamma \subset V$.  The contradiction shows
  $V=W$ and hence the uniqueness of $V$.

  To show that $\nu$, $r$, and $q$ are unique, suppose $0\le \mu\le
  s\le d$, $q_*\in \Nats$, and $B\in \Reals^{d\times d}$ is
  nonsingular, such that
  \begin{align} \label{e:lattice-2}
    |\chf_{B X}(2\pi u)|=1\Iff u\in\Ints^s \times \{0\},
  \end{align}
  and $B X = (Y_*, Z_*)$ with $Y_*\in \gamma + \Ints^s$, $Z^* \in
  \Reals^{d-s}$, and $\gamma = (0,\ldots, 0, \gamma_\mu, \gamma_{\mu+ 1},
  \ldots, \gamma_s)$, where $\gamma_\mu = p_*/q_*$ with $0\le p_* <
  q_*$ being coprime, and $\gamma_{\nu+1}, \ldots, \gamma_s\in (0,1)
  \setminus \Rats$ are rationally independent.  By $B' u\in\Gamma \Iff
  |\chf_{BX}(2\pi u)|=1$ and \eqref{e:lattice-2}, $\Gamma = B'(\Ints^s
  \times \{0\})$.  Since $B$ is nonsingular, a comparison of
  dimensions yields $s=\dim(V)=r$.
  
  Let $r\ge 1$, otherwise nothing remains to be shown.  Then by
  $\Gamma = T'(\Ints^r\times \{0\}) = B'(\Ints^r\times \{0\})$,
  $T'_1\Ints^r = B'_1\Ints^r$, where $T_1, B_1\in \Reals^{r\times 
    d}$ consist of the first $r$ rows of $T$ and $B$, respectively.
  Then there are $J, J_*\in\Ints^{r\times r}$ such that $T_1 = J B_1$
  and $J_* T_1 = B_1$, giving $J_* J B_1 = J_* T_1 = B_1$.  Since the
  rows of $B_1$ are linearly independent, $J_* J=\Id_r$.  Thus $J^{-1}
  = J_*$.  On the other hand, $B_1 X = Y_*$.  Then $T_1 X = J B_1 X =
  J Y_*\in J(\gamma + \Ints^r) = J\gamma + J\Ints^r = J\gamma +
  \Ints^r$.  Since $T_1 X = Y \in \beta + \Ints^r$, then  $\beta -
  J\gamma = (\eno b r) \in\Ints^r$.  Let $J=(g_{ij})$.  Each $\beta_i
  = c_i + g_{i,\mu+1} \gamma_{\mu+1} + \ldots + g_{ir}\gamma_r$ with
  $c_i = b_i + g_{i1} \gamma_1 + \ldots + g_{i\mu}\gamma_\mu  = b_i +
  g_{i\mu}\gamma_\mu\in \Ints$.  Since $\beta_i$, $i>\nu$, are
  rationally independent, this leads to $\mu\le \nu$.  Likewise,
  $\nu\le \mu$.  Thus $\mu = \nu$.  For $i\le \nu$, $g_{i,\nu+1}
  \gamma_{\nu+1} + \cdots + g_{ir} \gamma_r = \beta_i - c_i\in\Ints$.
  By rational independence of $\gamma_{\nu+1}$, \ldots, $\gamma_r$,
  $\beta_i = c_i$.  In particular, $\beta_\nu = k\gamma_\nu - l$ with
  $k = g_{\nu\nu}$ and $l = -b_\nu$.  Likewise, $\gamma_\nu = k_*
  \beta_\nu - l_*$ with $k_*, l_*\in\Ints$.  As a result $(kk_*-1)
  \beta_\nu = k(\gamma_\nu + l_*) - \beta_\nu = l + k l_*\in\Ints$.
  Since $\beta_\nu = p/q$ with $0\le p <q$ being coprime, $q\gv k k_*
  -1$, so $k_*$ and $q$ are coprime.  Then $\gamma_\nu =p_*/q$ with
  $p_* = k_* p - l_*q$ being coprime with $q$.  Thus $q_* = q$,
  completing the proof.
\end{proof}

\begin{proof}[Proof of Proposition \ref{p:aperiod}]
  (1) Let $K\xi = (\eno\zeta {\nu-1}, p + q \zeta_\nu)$, where
  $K\in\Ints^{\nu\times \nu}$ with $\det K=\pm 1$, $0\le p < q$ are
  coprime, and $\zeta = (\eno \zeta \nu)\in\Ints^\nu$ is strongly
  aperiodic.  By $\xi\in\Ints^\nu$, $\ip t\xi\in\Ints$ for
  $t\in\Ints^\nu$.  Conversely, if $\ip t \xi\in \Ints$, then
  letting $s= (K')^{-1} t$,  $\ip s {K\xi} = s_1\zeta_1 + \cdots +
  s_{\nu-1} \zeta_{\nu-1} + q s_\nu\zeta_\nu + p s_\nu\in \Ints$.  The
  strong aperiodicity of $\zeta$ implies $\eno s {\nu-1}, q s_\nu, p
  s_\nu\in\Ints$.  Since $p$ and $q$ are coprime, then $s_\nu\in
  \Ints$.  Thus $s\in \Ints^\nu$ and $t = K' s\in \Ints^\nu$.  This
  shows $\xi$ is aperiodic. 

  Conversely, let $\xi$ be aperiodic.  Define $\Gamma = \Gamma_\xi$ as
  in \eqref{e:Gamma-lattice}.  Then $\Gamma$ is a
  lattice.  Since $\Ints^\nu \subset \Gamma$, by Smith normal form
  (\cite {tao:06:cup}, Th.~3.7), there are linearly independent
  $\eno u\nu\in \Gamma$ and integers $1\le n_1 \le \cdots \le n_\nu$
  with $n_i\gv n_{i+1}$, such that, letting $M = (\eno u\nu)$ and $D =
  \diag(\eno n\nu)$,
  \begin{align} \label{e:aperiodic-lattice}
    \Gamma = M\Ints^\nu, \quad
    \Ints^\nu = M D \Ints^\nu.
  \end{align}
  By $u_i\in\Gamma$, $\ip {u_i} \xi \in s_i + \Ints$ for some $s_i$.
  In matrix form, $M'\xi \in s + \Ints^\nu$, where $s=(\eno s \nu)$.
  Define $K = D M'$, $Z = K\xi$, and $b = D s$.  From the second
  identity in \eqref {e:aperiodic-lattice}, $\Ints^\nu = K'\Ints^\nu$,
  giving $K, K^{-1}\in \Ints^{\nu\times \nu}$.  Then $Z\in\Ints^\nu$
  and is aperiodic.  Meanwhile, $Z = D M'\xi\in D(s + \Ints^\nu) = b +
  D \Ints^\nu$.  Then from $Z\in (b + D\Ints^\nu)\cap \Ints^\nu$ and
  $D \Ints^\nu\subset \Ints^\nu$, $b\in \Ints^\nu$.  Let $\inum Z$,
  $\inum {Z'}$ be \iid $\sim Z$.  For $m,n\ge 0$, $S_m(Z) - S_n(Z')\in
  \Ints b + D\Ints^\nu \subset \Ints^\nu$.  By aperiodicity of $Z$,
  for every standard base vector $e_i$ of $\Reals^\nu$, there are $m$
  and $n$ such that $\pr{S_m(Z) - S_n(Z') = e_i}>0$ (\cite{\Spit},
  p.~20).  This yields $e_i \in\Ints b + D \Ints^\nu$.  As a result,
  $\Ints b+ D\Ints^\nu = \Ints^\nu$.  Let $s_i\in\Ints$ and $v_i =
  (v_{i1}, \ldots, v_{i\nu})\in\Ints^\nu$, such that $b s_i + D v_i =
  e_i$.  Write $b = (\eno b\nu)$.  By comparing the coordinates,
  \begin{align*}
    b_i s_i + n_i v_{ii} = 1, \quad b_j s_i + n_j v_{ij}=0,\ j\ne i.
  \end{align*}
  Thus, each pair of $b_i$ and $n_i$ are coprime.  For $j>i$, as
  $n_j\ne 0$, $n_j\gv b_j s_i$, so $n_j \gv s_i$.  Then by  $b_i
  (s_i/n_j) n_j + n_i  v_{ii}=1$, $n_j$ and $n_i$ are coprime.  By
  $n_i\gv n_j$, this gives $n_i=1$.  As a result, $n_1 = \cdots =
  n_{\nu-1}=1$.  Put $q=n_\nu$ and let $0\le p<q$ such that $q\gv (b_\nu -
  p)$.  Let $\zeta = D^{-1}(Z - p e_\nu)$.  By $Z\in b + D\Ints^\nu$
  and $b- p e_\nu \in D\Ints^\nu$, $\zeta\in\Ints^\nu$.  If $\ip
  t\zeta\in s+\Ints$, where $t\in\Reals^\nu$ and $s\in\Reals$, then 
  $\ip {M t} \xi = \ip t {M'\xi} = \ip t {D^{-1} Z} \in c + \Ints$
  with $c = s + p\ip t {D^{-1} e_\nu}$.  Then $M t\in \Gamma
  = M \Ints^\nu$, so $t\in\Ints^\nu$.  Thus $\zeta$ is strongly
  aperiodic.  By $K\xi = Z = p e_\nu + D\zeta$, the proof is
  complete.

  (2) Let $\zeta = L - \beta_\nu e_\nu$.  Then $\zeta \in \Ints^\nu$.
  Since for $u\in\ssp$, $|\chf_{TX}(2\pi u)|=1\Iff u\in \Ints^\nu
  \times \{0\}$, then for $v\in \Reals^\nu$, $|\chf_\zeta(2\pi v)| =
  |\chf_L(2\pi v)|=1\Iff v\in\Ints^\nu$, so $\zeta$ is strongly
  aperiodic.  Then by (1), $D L= (\eno\zeta {\nu-1}, p + q
  \zeta_\nu)\in\Ints^\nu$ is aperiodic.
\end{proof}

\begin{proof}[Proof of Lemma \ref{l:normalize}]
  Recall $\Rats$, $\Reals$, and their quotient $\Reals/\Rats$ are
  vector spaces over the field $\Rats$,   Let $\bar a = (\eno{\bar a}
  r)$ with $\bar a_i = a_i + \Rats\in \Reals/\Rats$.  First, if $\bar
  a\ne 0$, then there are linearly independent $\eno{\bar u} s\in
  \Reals/\Rats$, $1\le s\le r$, such that $\bar a= A \bar u$, where $A
  \in \Rats^{r\times s}$ is of full column rank and $\bar u=(\eno{\bar
    u} s)$.  Equivalently, $a - A u\in \Rats^r$.  Note that $u_i$ are
  rationally independent.  By multiplying $A$ by a large $m\in\Nats$
  and dividing $u$ by $m$, $A$ can be assumed to be in $\Ints^{r\times
    s}$.  It is known that there are $P\in\Ints^{r\times r}$ and $R\in
  \Ints^{s\times s}$ with $|\det P| = |\det R|= 1$, such that $P A
  = \binom{D}{O} R$, where $D=\diag(\eno d s)$ with $d_i\in\Nats$ and
  $d_i\gv d_{i+1}$ (cf.\ \cite {jacobson:75b:sv-ny}, Th.~III.5).  Let
  $D R u = v$.  Then $P (a -A u) = P a - (v, 0)\in\Rats^r$, so $P a =
  (\tilde v, w)$, where $\tilde v = v + y$ for some $y\in \Rats^s$ and
  $w \in \Rats^{r-s}$.  The coordinates of $\tilde v$ are rationally
  independent.  On the other hand, similar to $A$, there is
  $M\in\Ints^{(r-s)\times (r-s)}$ with $\det M = \pm 1$, such that $M
  w  = (q, 0, \ldots, 0)\in \Rats^{r-s}$ with $q\ge 0$.  Then $K_0 =\!
  \binom{\!\!\Id_s{\ }}{\hphantom{\Id_s}M} P$ gives $K_0 a = (\tilde
  v, q, 0, \ldots, 0)$.  By permuting the coordinates, the lemma
  follows.  Finally, if $\bar a=0$, then $a\in \Rats^r$.  Following
  the treatment of the above $w$, the result follows.
\end{proof}

\section{Proofs regarding distributions in the domain of attraction}

\label{s:mv-stable}

\begin{proof}[Proof of \eqref{e:prelim-V} and \eqref{e:prelim-q-V}]
  Let $X\sim F\in \DA(\alpha)$.  If $\alpha=2$, then \eqref
  {e:prelim-q-V} is part of \cite{\Rva}, Th.~4.1.  For any $c>1$,
  $V_X(s) \le V_X(c s) \le V_X(s) + c^2 s^2 q_X(s)$, which by  \eqref
  {e:prelim-q-V} gives $V_X(c s)/V_X(s)\to 1$ as $s\toi$.  Then \eqref
  {e:prelim-V} follows.  If $\alpha\in (0,2)$, then Th.~4.2 of \cite
  {\Rva} states that $q_X\in \RV_{-\alpha}$, which leads to both
  \eqref {e:prelim-V} and \eqref{e:prelim-q-V} (cf.~\cite{\Bing},
  Th.~1.6.4).
\end{proof}

\begin{proof}[Proof of \eqref{e:prelim-WC}]
  For the univariate case, see \cite{\Bing}, p.~347.  For the
  multivariate case, first, let $\alpha\in (0,2)$.  The proof of 
  Th.~4.2 of \cite{\Rva} shows that a choice of $a_n$ is the infimum
  of all $s$ such that
  \begin{align*}
    \pr{|X|>s, X/|X|\in E} \le \gamma(E)/n \le \pr{|X|\ge s,
    X/|X|\in E},
  \end{align*}
  where $\gamma$ is a nonzero measure on $S^{d-1}$ and $E$ is any
  fixed subset of $S^{d-1}$ with $\gamma(E)>0$.  By Th.~14.10 of
  \cite{\Sato}, $\gamma$ is finite.  Letting $E = S^{d-1}$ and
  $c = \gamma(S^{d-1})$, it follows that $a_n$ can be any $s$
  satisfying $q_X(s) \le c/n \le q_X(s-)$.  Then by \eqref
  {e:prelim-q-V} and \eqref{e:prelim-A}, $a_n$ can (also) be taken to
  be any sequence such that $A(a_n)\sim cn$.

  Let $\alpha=2$ and $b(u) = \ip u {\Sigma u}$, where $\Sigma$ is the
  covariance matrix of the limiting normal distribution.  If $\mean
  |X|^2<\infty$, then \eqref{e:prelim-WC} follows from the Central
  Limit Theorem.  Suppose $\mean |X|^2 = \infty$.  By Th.~2.4 of
  \cite{\Rva}, $a_n$ can be any sequence such that for any $\rx>0$,
  (i) $n q_X(\rx a_n)\to 0$ and (ii) $(n/a^2_n) [m_V(\rx a_n,u) -
  \ip{c_V(\rx a_n)} u^2]\to b(u)$ for any $u\in S^{d-1}$.  Since
  $|c_V(s)|^2 = o(V_X(s))$ as $s\toi$ (\cite{\Rva}, (4.5)), by
  \eqref {e:prelim-V} and \eqref{e:m-V-normal}, (ii) is equivalent to
  $n/A(a_n) \to \sum_i b(e_i)$.  Once (ii) is satisfied, by \eqref
  {e:prelim-q-V}, (i) is satisfied.  Then the claim on $a_n$ follows.
\end{proof}

\begin{proof}[Proof of \eqref{e:centering}]
  For the univariate case, see \cite{\Bing}, p.~347.  For the
  multivariate case, according to the last comment on p.~190 in
  \cite{\Rva}, $b_n$ can be taken to be $(n/a_n) c_X(t a_n) + \gamma$,
  where $\gamma$ is any constant vector, and $t>0$ is any fixed number
  such that $\{|x|=t\}$ has measure 0 under the \levy measure of the
  limiting stable law.   From the characterization of the \levy
  measure (cf.\ \cite{\Rva}, (3.4)--(3.5)), $t$ can be any
  positive number.  It follows that any $b_n$ satisfying \eqref
  {e:prelim-WC} must be of the form $(n/a_n) c_X(a_n) + \gamma +
  \rx_n$ for some constant vector $\gamma$, where $\rx_n\to 0$ as
  $n\toi$.   This implies \eqref{e:centering}.
\end{proof}

\end{document}